\documentclass[11pt]{article}

\usepackage[utf8]{inputenc}

\usepackage{geometry}
\usepackage{tikz}

\usepackage{float}

\usepackage{graphicx}

\usepackage{booktabs} 
\usepackage{array} 
\usepackage{paralist} 
\usepackage{verbatim} 
\usepackage{graphics}
\usepackage{subfigure}
\usepackage{epstopdf}
\usepackage{tabularx}
\usepackage{fancyhdr} 
\usepackage{color}
\usepackage[title]{appendix}

\usepackage[nottoc,notlof,notlot]{tocbibind}
\usepackage[titles,subfigure]{tocloft}

\usepackage{amsmath,amsfonts,amsthm,mathrsfs,amssymb,cite}
\usepackage{cases}

\newcommand{\R}{{\mathbb R}}
\newcommand{\Z}{{\mathbb Z}}

\newcommand{\C}{{\mathbb C}}
\newcommand{\E}{{\mathcal E}}

\newcommand{\be}{\begin{eqnarray}}
\newcommand{\ben}{\begin{eqnarray*}}
\newcommand{\en}{\end{eqnarray}}
\newcommand{\enn}{\end{eqnarray*}}
\newcommand{\ba}{\backslash}
\newcommand{\pa}{\partial}

\newcommand{\ov}{\overline}

\newcommand{\G}{\Gamma}

\newcommand{\vep}{\varepsilon}
\newcommand{\Om}{\Omega}
\newcommand{\om}{\omega}
\newcommand{\sig}{\sigma}

\newcommand{\n}{\noindent}

\newtheorem{thm}{Theorem}[section]

\newtheorem{lem}{Lemma}[section]

\theoremstyle{definition}
\newtheorem{defn}{Definition}[section]
\theoremstyle{remark}

\newtheorem{rem}{Remark}[section]
\numberwithin{equation}{section}

\title{\bf Well-posedness and convergence analysis of PML method for time-dependent acoustic scattering problems over a locally rough surface}

\author{ Hongxia Guo\thanks{School of Mathematical Sciences and LPMC, Nankai University, 300071 Tianjin, China. ({\tt hxguo@nankai.edu.cn})}  \and Guanghui Hu\thanks{(Corresponding author) School of Mathematical Sciences and LPMC, Nankai University, 300071 Tianjin, China. ({\tt ghhu@nankai.edu.cn})}
}

\date{} 
\begin{document}
\maketitle

\begin{abstract}
  We aim to analyze and calculate time-dependent acoustic wave scattering by a bounded obstacle and a locally perturbed non-selfintersecting curve. The scattering problem is  equivalently reformulated as an initial-boundary value problem of the wave equation in a truncated bounded domain through a well-defined transparent boundary condition.
Well-posedness and stability of the reduced problem are established. Numerically, we adopt the perfect matched layer (PML) scheme for simulating the propagation of perturbed waves.
By designing a special absorbing medium in a semi-circular PML, we show well-posedness and stability of the truncated initial-boundary value problem. Finally, we prove that the PML solution converges exponentially to the exact solution in the physical domain. Numerical results are reported to verify the exponential convergence with respect to absorbing medium parameters and  thickness of the PML.

\vspace{.2in} {\bf Keywords: wave equation, well-posedness, PML, convergence.
}

\end{abstract}

\section{Introduction}

The scattering problems over a half-space with local perturbations have widely considered in the fields of radar techniques, sonar, ocean surface detection, medical detection, geophysics, outdoor sound propagation and so on. Such problems are also referred to as cavity scattering problems in the literature; see e.g. \cite{Lidtn, Wood, Amm00, Amm03} where variational and integral equation methods (see also \cite{Amm00, Amm03})
were adopted to reduce the unbounded physical domain to a truncated computational domain in the time-harmonic regime. In this paper we concern the time-dependent scattering problems governed by wave equations.

If the domain of the wave equations is unbounded, one can either use
 transparent/absorbing boundary conditions  to minimize the spurious reflections or absorbing boundary layers, which are usually referred to perfectly matched layers (PML),  to bound the unbounded physical domain by truncated computational domain in the numerical simulation.
 A major challenge is to construct the temporal dependence of the transparent boundary condition \cite{Nedelec} or the artificially designed absorbing medium (see e.g.,\cite{Joly06}) in the PML method.
The PML scheme is initially introduced by B\'{e}renger for 2D and 3D Maxwell equations \cite{Ber94, Ber96}. The basic idea of the PML is to surround the physically computational domain by some {\rm artificial} medium that absorbs outgoing waves effectively.
Mathematically, a PML layer method can be equivalently formulated as a complex stretching of the external domain.
Such a feature makes PML an effective for modeling a variety of wave phenomena \cite{Bram, Chew94, Coll, Tur}.
Due to its barely reflective absorption of out going waves,  PML turns out to be very popular for simulating the propagation of waves in time domain \cite{hoop, Joly03, Joly12}.

For time-harmonic scattering problems,
the PML formulation was introduced in \cite{Cheny} to locally perturbed rough surface scattering problems. We also refer the readers to  \cite{Chen10-fre, Hohage, Lassas} for the  analysis of acoustic scattering problem in the whole space and to \cite{Bram08, Bao05}
for electromagnetic scattering problems where the convergence rate depends exponentially on the absorption parameter and thickness of PML layer.
In theory it is crucial to investigate well-posedness, stability, convergence of the PML formulation. This paper is concerned with the mathematical analysis  and numerical simulation of the time-dependent acoustic scattering problem in a locally perturbed half space with the following issues:
\begin{description}
\item[(1)] well-posedeness and stability of the time-dependent problem using the Dirichlet-to-Neumann (DtN) operator;
\item[(2)] well-posedness and long-time stability of the PML formulation in a truncated domain;
\item[(3)] convergence of the solution of the PML formulation to that of the original problem;
\item[(4)] numerical tests of the exponential convergence of the PML method.
\end{description}



To the best of our knowledge, the mathematical investigation of the convergence/error analysis of the PML problem for wave equations is far from being complete, in comparision with the vast works for time-harmonic scattering problems.
Existing results mainly concern the well-posedness and stability of PML problem; see e.g., \cite{appelo06, joly02,joly03,Bramble} where the absorption parameter was all assumed to be a constant.
Using Laplace transform and the transparent boundary conditions (TBC),
the exponential convergence with respect to the thickness and the absorbing parameter has been justified in \cite{Chen09,Chen12} for time-dependent acoustic scattering problems in the whole space. Later the approach of applying Laplace transform \cite{Chen09,Chen12} has been extended to the cases of waveguides \cite{Becache21}, periodic structures as well as electromagnetic scattering problems in the whole space \cite{Wei20}. See
also \cite{bao18, li15, wei19} for the analysis of the time-dependent fluid-solid interaction problems and electromagnetic scattering problems.
Nevertheless, the exponential convergence results in the aforementioned literatures are not confirmed by numerical examples. On the other hand, as far as we know, a comprehensive analysis is still missing for the PML method to the acoustic wave equation in a locally perturbed half space.

In this work, a perturbation of the half plane $\{x: x_2>0\}$ can be caused by either a bounded obstacle imbedded in the background medium or a compact change of the unbounded curve $x_2=0$; see the geometry shown in Figure
\ref{fig1}. Firstly, we adopt the approach of \cite{Chen09}
to prove well-posedness of the scattering  problem in proper time-dependent Sobolev spaces by using a well-defined TBC (Dirichlet-to-Neumann operator). We complement the earlier work \cite{Chen09} by describing mapping properties of the DtN operator and by connecting the TBCs defined over a finite and an infinite time period, which seem not well-addressed in the literature. 
Motivated by \cite{Chen09}, a circular PML layer with special medium properties will then be defined to truncate the original problem.   
A first order symmetric hyperbolic system is derived for the truncated PML problem, which is similar to those considered in \cite{Chen09, Chen12, Bao18}. The well-posedness and stability  of the truncated PML problem are justified by Laplace transform, variational method together with the energy method of \cite{joly02}.
The convergence of the PML scheme is based on the stability estimate of an initial-boundary value problem in the PML layer and  the exponential decay of the PML extension problem to be proved using modified Bessel functions. Such a technique is also inspired \cite{Chen09}.

This paper is organized  as follows. In the subsequent Section 2,  we first introduce the mathematical model and rigorous define the transparent boundary condition (TBC) to reformulate the scattering problem to an initial-boundary value problem in a truncated bounded domain. Well-posedness and stability will then be shown in Section 2.2. In Section 3, we derive a PML formulation in the half plane by complex coordinates stretching inspired by \cite{Chew94, Chen09, Chen12, Petropoulos} and study the well-posedness and stability  for the PML problem. We analyze the  exponential convergence of the PML method in the half space in Section 4. In the final Section 5, two numerical examples are reported to show the performance of the PML method.

\section{Mathematical formulations}
Let $\Gamma_0$ be a local perturbation of the straight line $\{(x_1,0): x_1\in \R\}$ such that $\Gamma_0$ coincides with $x_2=0$ in $|x_1|>R$ for some $R>0$ and that $\Gamma_0$ is a non-selfintersecting $C^2$-smooth curve. Denote by $\Omega\subset \R^2$ the unbounded domain above $\Gamma_0$, which is supposed to be filled by a homogeneous and isotropic medium with the unit mass density. Let $D \subset B_R^+:=\{x\in\Omega: |x|<R\}$ be  a bounded domain with the Lipschitz boundary $\pa D$ such that the exterior of $D$ is connected; see Figure \ref{fig1}.
Physically, the domain $D$ represents a sound soft obstacle embedded in $\Omega$. Write $\R^2_+=\{x\in \R^2: x_2>0\}$, $\Gamma_R^+:=\{x\in \Omega:  |x|=R\}$. It is obvious that $B_R^+$ is a Lipschitz domain.

The time-dependent acoustic scattering problem with the Dirichlet boundary condition enforcing on the obstacle $\partial D$ and the locally perturbed rough surface $\Gamma_0$ can be governed by the initial-boundary value problem of the wave equation
\be \label{eqs:wave}
\left\{\begin{array}{lll}
\partial_{t}^2 u(x,t)-\Delta u(x,t)=\partial_t f(x,t) &&\mbox{in}\quad (\Omega\backslash\overline{D})\times (0,T),\\
u(x,t)=0 \quad  &&\mbox{on}\quad ( \partial D \cup \Gamma_0)\times (0,T),\\
u(x,0)=\partial_t u(x,0)=0 &&\mbox{in}\quad (\Omega\backslash\overline{D}).
       \end{array}\right.
\en
Here, $T>0$ is an arbitrarily fixed positive number, the function $f$ represents an acoustic source term  compactly supported in $B_R^+\backslash\overline{D}$ and $u$ denotes the total field. In the exterior of $B_R^+$, the total field $u=u^{in}+u^{re}+u^{sc}$ can be divided into the sum of the incident field $u^{in}$, the reflected field $u^{re}$ corresponding to the unperturbed scattering problem in the homogeneous half space $x_2>0$ and the scattered field $u^{sc}$ caused by $D$ and the perturbation of the straight line $x_2=0$.  The first two components of $u$ will be explained as follows.

\begin{figure}[h]
\centering
\includegraphics[scale=0.4]{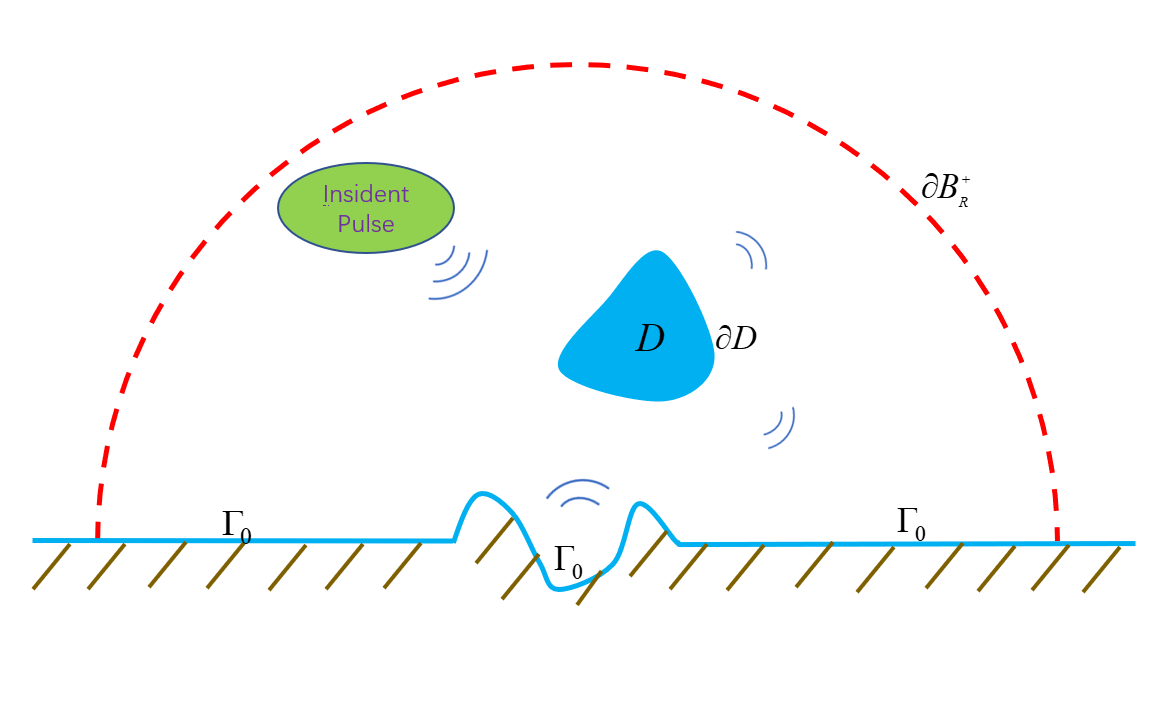}
\caption{Geometry of  acoustic wave scattering problem caused by a bounded obstacle $D$ and a locally perturbed curve $\Gamma_0$.
} \label{fig1}
\end{figure}
The incident field $u^{in}$ is generated by the inhomogeneous wave equation in $\R^2$:
\ben
\left\{\begin{array}{lll}
\partial_{t}^2 u(x,t)-\Delta u(x,t)=\partial_t f(x,t) &&\mbox{in}\quad \R^2, \; t>0,\\
u(x,0)=\partial_t u(x,0)=0 &&\mbox{in}\quad \R^2.
       \end{array}\right.
\enn
Obviously, the incident field $u^{in}$ takes the explicit form
\ben
u^{in}(x,t)=\int_{\R^2} G(x,t;\,y)*\pa_t f(y,t)\,dy \quad \mbox{in} \quad \R^2\times \R^+,
\enn
where $*$ denotes convolution between  $G$ and $\pa_t f$ with respect to the time $t$ , and
\ben
G(x,t;\,y):=\frac{H(t-|x-y|)}{2\pi\sqrt{t^2-|x-y|^2}},
\enn
is the Green's function of the wave operator $\pa_{t}^2-\Delta$ in the free space $\R^2\times\R$. Note that $H$ is the Heaviside function defined by
\ben
H(t):= \left\{\begin{array}{lll}
0,  && t\leq 0, \\
1, && t > 0.
\end{array}\right.
\enn
The reflected field $u^{re}$ caused by the incident field $u^{in}$ and the Dirichlet curve $x_2=0$ is governed by
\ben
\left\{\begin{array}{lll}
\partial_{t}^2 u^{re}(x,t)-\Delta u^{re}(x,t)=0 &&\mbox{in}\quad \R^2_+, \; t>0,\\
u^{re}(x,0)=\partial_t u^{re}(x,0)=0 &&\mbox{in}\quad \R^2_+,\\
u^{re}(x,t)=-u^{in}(x,t) &&\mbox{on}\quad x_2=0, t>0.
       \end{array}\right.
\enn
Denote by $y^{*}=(y_1,-y_2)$ the reflection of $y=(y_1,y_2)$ by the straight line $x_2=0$. Through simple calculations, we obtain the expression of the reflected field $u^{re}$ by
\ben
u^{re}(x,t)&=&-\int_{\R^2}G(x,t;\,y^{*})*\pa_t f(y,t)\,dy \\
&=& -\int_0^t \int_{B_R^+\backslash\overline{D}} G(x-y^*,t-\tau) \pa\tau f(y,\tau) \,dy d\tau.
\enn
Evidently, the sum $u^{in}+u^{re}$ denotes the total field to the unperturbed scattering problem that corresponds to $u^{in}$ and the Dirichlet curve $x_2=0$.
The function $u^{sc}$ consists of the scattered wave from the bounded domain $D$ and the local perturbation $\{x\in \Gamma_0:x_2\neq0, x_1\in\R  \}$.

Throughout this paper, we suppose that for any bounded domain $\Om_0$, $f\in H^2(0,T;L^2(\Om_0))$ and that $f|_{t=0}=0$, $f=\tilde{f}|_{(0,T)}$ where
\ben
\tilde{f}\in H^2(0,\infty;L^2(\Om_0)),\quad \|\tilde{f}\|_{ H^2(0,\infty;L^2(\Om_0))}\leq \|f\|_{ H^2(0,T;L^2(\Om_0))}.
\enn
This implies that the source term $\partial_t f$ on the right hand side of \eqref{eqs:wave} belongs to $H^1(0,T;\Omega\backslash\overline{D})$.
Hence, applying the approach of J. L. Lions (see  \cite[Theorem 8.1, Chapter 3]{LM72} and \cite[Theorem 8.2, Chapter 3]{LM72}) there exists a unique solution $u\in
C(0,T; H^1_0(\Omega\backslash \ov {D})) \cap C^1(0,T; L^2(\Omega\backslash \ov{D}))$ to \eqref{eqs:wave}.

\subsection{A transparent boundary condition (TBC) on a semi-circle }
The aim of this section is to rigorously address the Dirichlet-to-Neumann map for the wave equation \eqref{eqs:wave} in a locally perturbed half-plane. We shall follow the spirit of \cite{Chen09} for a bounded sound-hard obstacle but complement the definition of DtN there by describing mapping properties in time-dependent Sobolev spaces and connecting the DtN operators defined over a finite and an infinite time period.
More precisely, we shall define the time-domain boundary operator $\mathscr{T}$ by
\be \label{bc:operator}
\mathscr{T}u=\pa_r u \quad \mbox{on} \quad  \Gamma_R^+\times(0,T),
\en
which is called the TBC. Thus, the time-domain scattering problem (\ref{eqs:wave}) in the unbounded domain over the local rough surface can be  reduced into an equivalent initial-boundary value problem in the bounded domain $\Om_R^+:=B_R^+\backslash\overline{D}$:
\be \label{eqs:wave-b}
\left\{\begin{array}{lll}
\partial_{t}^2 u-\Delta u=\partial_t f &&\mbox{in}\quad \Om_R^+\times (0,T),\\
u=0 \quad  &&\mbox{on}\quad  (\partial D \cup  \Gamma_0)  \times (0,T),\\
\pa_r u=\mathscr{T}u \quad &&\mbox{on} \quad  \Gamma_R^+\times (0,T),\\
u|_{t=0}=\partial_t u|_{t=0}=0 &&\mbox{in}\quad \Om_R^+.
       \end{array}\right.
\en
In what follows we derive a representation of the boundary operator $\mathscr{T}$.
Let $H^{1/2}_0 (\Gamma_R^+)$, $H^{1/2} (\Gamma_R^+)$, $H^{-1/2} (\Gamma_R^+)$, $H^{-1/2}_0 (\Gamma_R^+)$ be Sobolev spaces defined on the open arc $\Gamma_R^+$. Then $H^{1/2}_0 (\Gamma_R^+)$ and $H^{-1/2} (\Gamma_R^+)$, $H^{1/2} (\Gamma_R^+)$ and $H^{-1/2}_0 (\Gamma_R^+)$ are anti-linear dual spaces\cite{Mclean}.
For $u\in H^1_0(B_R^+)$, we have $u|_{\Gamma_R^+} \in H^{1/2}_0 (\Gamma_R^+)$.

\noindent Consider an initial-boundary value problem  over a finite time period
\be \label{eqs:test1}
\left\{\begin{array}{lll}
\partial_{t}^2 u(x,t)-\Delta u(x,t)=0 &&\mbox{in}\quad (\Omega\backslash  \ov{ B_R^+})\times (0,T),\\
u(x,t)=g(x,t) \quad  &&\mbox{on}\quad \Gamma_R^+ \times (0,T),\\
u(x,t)=0 \quad &&\mbox{on}\quad (\Gamma_0 \cap |x|>R ) \times (0,T),\\
u(x,0)=\partial_t u(x,0)=0 &&\mbox{in}\quad (\Omega\backslash \ov {B_R^+})
       \end{array}\right.
\en
where $g\in C(0,T;H^{1/2}_0 (\Gamma_R^+))\,\cap C^1(0,T; H^{-1/2}(\Gamma_R^+))$ satisfying $g(x,0)=\partial_t g(x,0)=0$.
By \cite[Chapter 7]{Func},  there exists a  unique solution $u$ to the above equations satisfying
\ben
u\in C(0,T; H^1_{\diamond}(\Omega\backslash \ov {B_R^+})) \cap C^1(0,T; L^2(\Omega\backslash \ov{B_R^+})),
\enn
where $H^1_{\diamond}(\Omega\backslash \ov{ B_R^+})=\{u\in H^1(\Omega\backslash B_R^+): u=0 \;  \mbox{on} \;\{\Gamma_0  \cap   |x|>R\}\}$. 

\begin{defn}
 The DtN operator $\mathscr{T}: C(0,T;H^{1/2}_0 (\Gamma_R^+))\rightarrow C(0,T;H^{-1/2} (\Gamma_R^+))$ over a finite time period $(0,T)$ is defined as
\ben
\mathscr{T} g=\pa_r u \quad \mbox{on} \quad \Gamma_R^+\times(0,T),
\enn
where $u\in C(0,T; H^1_{\diamond}(\Omega\backslash \ov {B_R^+})) \cap C^1(0,T; L^2(\Omega\backslash \ov{B_R^+}))$ is the unique solution to (\ref{eqs:test1}).
\end{defn}
\noindent Consider another initial-boundary value problem but over an infinite time
\be \label{eqs:test2}
\left\{\begin{array}{lll}
\partial_{t}^2 w(x,t)-\Delta w(x,t)=0 &&\mbox{in}\quad (\Omega\backslash \ov{ B_R^+})\times (0,\infty),\\
w(x,t)=\tilde g(x,t) \quad  &&\mbox{on}\quad \Gamma_R^+ \times (0,\infty),\\
w(x,t)=0 \quad &&\mbox{on}\quad (\Gamma_0 \cap |x|>R ) \times (0,\infty),\\
w(x,0)=\partial_t w(x,0)=0 &&\mbox{in}\quad (\Omega\backslash \ov{B_R^+}),
       \end{array}\right.
\en
with the initial-boundary value
\be\label{reg}
\tilde{g}\in L^2(0,\infty;H^{1/2}_0 (\Gamma_R^+))\,\cap H^1(0,\infty; H^{-1/2}(\Gamma_R^+)),\quad \tilde{g}(x,0)=\partial_t \tilde{g}(x,0)=0.
\en
\begin{defn}
The DtN operator $\tilde{ \mathscr{T}}: L^2(0,\infty;H^{1/2}_0 (\Gamma_R^+))\rightarrow L^2(0,\infty;H^{-1/2} (\Gamma_R^+))$ over the infinite time period $(0,\infty)$ is defined as
\ben
\tilde{\mathscr{T}} \tilde g=\pa_r w  \quad  \mbox{on} \quad \Gamma_R^+\times(0,\infty),
\enn
where $ w\in L^2(0,\infty; H^1_{\diamond}(\Omega\backslash \ov {B_R^+})) \cap H^1(0,\infty; L^2(\Omega\backslash \ov{B_R^+}))$ is the unique solution of (\ref{eqs:test2}).
\end{defn}
\begin{lem} \label{dtn-t}
Let $\tilde g$ be the boundary value of the problem (\ref{eqs:test1})
and denote by $\tilde{g}$ its zero extension to $t>T$. Then
\ben
\tilde{\mathscr{T}} \tilde g=\mathscr{T} g \quad \mbox{in} \quad
L^2(0,T;H^{-1/2} (\Gamma_R^+)).\enn
\end{lem}
\begin{proof}It is obvious that the extension $\tilde{g}$ fulfills the regularity and initial values specified in \eqref{reg}.
Define $v:=u-w$, where $u$ and $w$ are the unique solutions to \eqref{eqs:test1} and \eqref{eqs:test2}, respectively. It then follows that
\be \label{eqs:test3}
\left\{\begin{array}{lll}
\partial_{t}^2 v(x,t)-\Delta v(x,t)=0 &&\mbox{in}\quad (\Omega\backslash \ov {B_R^+})\times (0,T),\\
v(x,t)=0 \quad  &&\mbox{on}\quad \Gamma_R^+ \times (0,T),\\
v(x,t)=0 \quad &&\mbox{on}\quad (\Gamma_0 \cap |x|>R ) \times (0,T),\\
v(x,0)=\partial_t v(x,0)=0 &&\mbox{in}\quad (\Omega\backslash \ov {B_R^+}).
       \end{array}\right.
\en
By uniqueness to the above system (see e.g., \cite{LM72, Func}), we get $v\equiv0$ in $(\Omega\backslash \overline{B_R^+})\times (0,T)$, implying that $\pa_r v=0$ on  $\Gamma_R^+ \times (0,T)$. Hence, we obtain that $\tilde{\mathscr{T}}\tilde g=\mathscr{T} g$ on $\Gamma_R^+\times (0, T)$.
\end{proof}

From the proof of Lemma \ref{dtn-t} we conclude that the definition of $\mathscr{T}g$ is independent of the values of $g$ in $t>T$.
Below we want to derive an expression of $\tilde{\mathscr{T}}$ by Laplace transform.
For any $s\in \C$ with $\textnormal{Re}(s)>0$, applying Laplace transform to (\ref{eqs:test2}) with respect to $t$, we see that $w_L=\mathscr{L}(w)$ satisfies the Helmoholtz equation
\be\label{w1}
-\Delta w_L +s^2 w_L=0 \quad \mbox{in}\quad \Om\ba  \ov{B_R^+},
\en
together with the radiation condition
\be\label{w2}
\sqrt{r}(\frac{\pa w_L}{\pa r}  +s w_L ) \rightarrow 0  \quad \mbox{as}\quad r=|x|\rightarrow\infty.
\en
Let $\mathscr{G}: H^{1/2}_0(\G_R^+)\rightarrow H^{-1/2}(\G_R^+)$ be the DtN operator in s-domain defined by
\ben
\mathscr{G}\tilde{g}_L=\pa_r w_L \quad \textnormal{on}\quad \Gamma_R^+.
\enn
where $w_L$ is the unique solution to \eqref{w1}-\eqref{w2} satisfying the boundary value $w_L=g_L$ on $\Gamma_R^+$ and $w_L=0$ on $\Gamma_0\cap \{x: |x|>R\}$.
Then it follows that $$\tilde{\mathscr{T}}=\mathscr{L}^{-1}\circ \mathscr{G} \circ \mathscr{L}.$$ Next, we derive a representation of the DtN operator $\mathscr G$. In the polar coordinates $(r,\theta)$,
$w_L$ can be expanded into the series (see e.g., \cite{Chen09, Lidtn, Wood})
\ben
w_L(r, \theta; s) =\sum_{n=1}^{\infty} \frac{K_n(sr)}{K_n(sR)}w^n _L(R,s)\sin{n\theta}, \quad r>R, \quad \theta\in[0,\pi],
\enn
where $$w^n _L(R,s)=\frac{2}{\pi} \int_{0}^{\pi} {w_L (R, \theta, s) \sin{n\theta} \, d\theta}= \frac{2}{\pi} \int_{0}^{\pi} {\tilde{g}_L (R, \theta, s) \sin{n\theta} \, d\theta}
.$$ Here $K_n(z)$ represents the modified Bessel function of order $n$. A simple calculation gives
\be \label{eq:gul}
\mathscr{G}w_L(R,\theta,s) = \frac{\pa w_L}{\pa r}|_{\Gamma_R^+}=s\sum_{n=1}^{\infty} \frac{K_n'(sR)}{K_n(sR)}w^n _L(R,s)\sin{n\theta}.
\en

The DtN operators $\mathscr{G}$ and $\mathscr{T}$ have the following properties.

\begin{lem} \label{lem:bd}
The operator $\mathscr{G}: H^{1/2}_0(\Gamma_R^+)  \rightarrow  H^{-1/2}(\Gamma_R^+)$ is bounded.
\end{lem}
\begin{proof}
By the recurrence formula of modified Bessel function
\ben
K_n^{'}(z)=-K_{n-1}(z)-\frac{n}{z}K_n(z),
\enn
we deduce that
\ben
\Big | \frac{K_n^{'}(sR)}{K_n(sR)} \Big | = \Big | \frac{n}{sR}+\frac{K_{n-1}(sR)}{K_n(sR)} \Big | \leq \frac{n}{|s|R}+1.
\enn
Let $|B_n|:=\Big | \frac{K_n^{'}(sR)}{K_n(sR)} \Big |$.  Then, $|B_n|\leq C\sqrt{n^2+1}$ for some constant $C>0$.
Given  $\phi \in H^{-1/2}_0(\Gamma_R^+)$, we expand
\ben
\phi(R,\theta)=\sum_{n=1}^{\infty}\phi_n(R)\sin{n\theta},\quad \phi_n(R)=\frac{2}{\pi}\int_{0}^{\pi}\phi(R,\theta)\sin n\theta \, d\theta.
\enn
By the definition of $\mathscr{G}$, for any $\om \in H^{1/2}_0(\Gamma_R^+)$ it follows that
\begin{align*}
\Big|\langle \mathscr{G}(\om),\,\phi \rangle_{\Gamma_R^+}\Big|&= \Big| \int_{\Gamma_R^+} s \sum_{n=1}^{\infty} \frac{K_n^{'}  (sR)}{K_n(sR)}\om_n(R) \sin n\theta  \sum_{n=1}^{\infty} \ov{\phi}_n(R) \sin n\theta\,d \gamma \Big| \\
&=\Big| sR \sum_{n=1}^{\infty} \frac{K_n^{'}  (sR)}{K_n(sR)}\om_n(R) \ov{\phi}_n(R) \int_0^\pi \sin^2 n\theta \, d\theta \\
&=\Big|\frac{\pi}{2}sR\sum_{n=1}^{\infty} \frac{K_n^{'}  (sR)}{K_n(sR)}\om_n(R) \ov{\phi}_n(R) \Big| \\
&\leq \frac{\pi}{2}|s|R \Big(\sum_{n=1}^{\infty} \Big| \frac{K_n^{'}  (sR)}{K_n(sR)} \Big|   |\om_n(R)|^2   \Big)^{1/2} \Big(\sum_{n=1}^{\infty} \Big| \frac{K_n^{'}  (sR)}{K_n(sR)} \Big|     | \phi_n(R)|^2    \Big)^{1/2}\\
&\leq C\Big( \sum_{n=1}^{\infty} \sqrt{1+n^2} |\om_n(R)|^2 \Big)^{1/2}   \Big( \sum_{n=1}^{\infty} \sqrt{1+n^2} |\phi_n(R)|^2 \Big)^{1/2} \\
&\leq C \| \om \|_{H^{1/2}_0(\Gamma_R^+)} \, \| \phi \|_{H^{1/2}_0(\Gamma_R^+)}.
\end{align*}
Then, we have
\ben
\|\mathscr G \om\|_{H^{-1/2}(\Gamma_R^+)} =\sup_{\phi\in H^{1/2}_0(\Gamma_R^+)}\frac{\Big|\langle \mathscr{G}(\om),\,\phi \rangle_{\Gamma_R^+}\Big|}{\|\phi\|_{H^{1/2}_0(\Gamma_R^+)}}\leq C \| \om\|_{H^{1/2}_0(\Gamma_R^+)}.
\enn
\end{proof}

\begin{lem} \label{lem:G}
It holds that, for any $\om \in H^{1/2}_0(\Gamma_R^+)$,
\ben
-\textnormal{Re}\,\langle s^{-1}\,\mathscr{G}\om,\, \om \rangle_{\Gamma_R^+} \geq0.
\enn
\end{lem}

\begin{proof}
Given $\om \in H^{1/2}_0(\Gamma_R^+)$, we have
\ben
\om(R,\theta)=\sum_{n=1}^{\infty}\om_n(R)\sin{n\theta},\quad \om_n(R)=\frac{2}{\pi}\int_{0}^{\pi}\om(R,\theta)\sin n\theta \, d\theta.
\enn
It follows from the expression of $\mathscr{G}u_L$ (\ref{eq:gul}) and Lemma \ref{lem:mbf-1}, we obtain
\begin{align*}
-\textnormal{Re}\,\langle s^{-1}\,\mathscr{G}\om,\om \rangle_{\Gamma_R^+}
=&-\textnormal{Re} \int_{\Gamma_R^+} \sum_{n=1}^{\infty} \frac{K_n^{'}  (sR)}{K_n(sR)}\om_n(R) \sin n\theta  \sum_{n=1}^{\infty} \ov{\om}_n(R) \sin n\theta\,d \gamma\\
=& -R \sum_{n=1}^{\infty}\textnormal{Re}\Big(\frac{K_n^{'}  (sR)}{K_n(sR)}\Big)|\om_n (R)|^2 \int_0^{\pi} \sin^2 n\theta \, d \theta\\
=&-\frac{\pi}{2} R \sum_{n=1}^{\infty}\textnormal{Re}\Big(\frac{K_n^{'}  (sR)}{K_n(sR)}\Big)|\om_n (R)|^2\geq 0.
\end{align*}
\end{proof}
Below we write the Laplace transform variable as $s=s_1+is_2$ with $s_1>0$, $s_2\in\R$.

\begin{lem} \label{lem:t}
Let $\om \in C(0,T;H^{1/2}_0({\Gamma_R^+})\cap C^1(0,T;H^{-1/2}({\Gamma_R^+}) $ with the initial values $\om(\cdot,\, 0)=\partial_t \om(\cdot,\,0)=0$. 
Then it holds that
\ben
\textnormal{Re}\, \int_0^T    e^{-2s_1 t}\langle \mathscr{T} \om, \, \partial_t \om \rangle_{\Gamma_R^+}\,dt \leq0.
\enn
\end{lem}
\begin{proof}
Let $\tilde{\om}(r,t)$ be the zero extension of $\om(r,t)$ with respect to $t$ in $\R$. Applying the Parseval identity (\ref{PI}) and Lemmas \ref{dtn-t} and \ref{lem:G}, we obtain
\begin{align*}
\textnormal{Re}\, \int_0^T e^{-2s_1 t}\langle \mathscr{T}\om, \, \partial_t \om \rangle_{\Gamma_R^+}\,dt
&=\textnormal{Re}\, \int_0^T e^{-2s_1 t} \int_{\Gamma_R^+}\mathscr{T}\om \partial_t \overline\om \,d\gamma\,dt \\
&=\textnormal{Re}\, \int_{\Gamma_R^+}\int_0^{\infty} e^{-2s_1 t} \tilde{ \mathscr{T}}\tilde\om \partial_t \overline{\tilde\om} \,dt \,d\gamma \\
&=\frac{1}{2\pi}\int_{-\infty}^{\infty}\textnormal{Re} \langle \mathscr{G}  \tilde{\om}_L,\, s\tilde{\om}_L \rangle_{\Gamma_R^+} \,d s_2 \\
&=\frac{1}{2\pi}\int_{-\infty}^{\infty} |s|^2 \textnormal{Re} \langle s^{-1}\mathscr{G}  \tilde{\om}_L,\, \tilde{\om}_L \rangle_{\Gamma_R^+} \,d s_2 \\
&\leq0.
\end{align*}
\end{proof}

\subsection{Well-posedness in the time-domain}
In the subsection, we prove well-posedness  of the truncated initial-boundary value problem (\ref{eqs:wave-b}) in the  bounded domain $\Om_R^+$ by using the variational method in the Laplace domain. Taking Laplace transform of (\ref{eqs:wave-b}) and using $f(\cdot, 0)=0$ we obtain
\be  \label{eq:s}
\left\{\begin{array}{lll}
\Delta u_L-s^2 u_L=sf_L  &&\mbox{in} \quad \Om_{R}^+\backslash\overline{D}, \\
    u_L=0   &&\mbox{on} \quad \partial D\cup\Gamma_0, \\

\partial_r u_L=\mathscr{G}u_L  &&\mbox{on} \quad \Gamma_R^+.
\end{array}\right.
\en
We formulate the variational formulation of problem (\ref{eq:s}) and show its well-posedness in the space $X_R:=\{u\in H^1(\Om_R^+):u=0\; \mbox{on} \;\partial D \cup \Gamma_0\}$.
Multiplying the Helmholtz equation in  (\ref{eq:s}) by the complex conjugate of a test function $v\in X_R$, applying the Green's formula with the boundary conditions on $\Gamma_R^+\cup\Gamma_0\cup\partial D$, we arrive at
\be \label{eq:s-v}
a(u_L,\,v)=\int_{\Om_R^+} f_L \,  \overline v \, d x  \quad \mbox{for all}\quad v\in X_R,
\en
where the sesquilinear form $a(\cdot,\cdot)$ is defined as
\ben
a(u_L,\,v)=\int_{\Om_R^+} \left( \frac{1}{s} \nabla u_L \cdot \nabla \overline v +s u_L  \overline v\right) \,d x -\langle s^{-1}\mathscr{G}u_L,\, v \rangle_{\Gamma_R^+}.
\enn

\begin{lem}\label{lem:wave-s}
The variational problem (\ref{eq:s-v}) has a unique solution $u_L\in X_R$ with the following stability estimate
\be \label{eq:wave-s}
\| \nabla u_L \|_{L^2(\Omega_R^+)} + \| s u_L \|_{L^2(\Omega_R^+)} \leq C \frac{(1+|s|)|s|}{s_1}\|f_L\|_{L^2(\Omega_R^+)},
\en
where $C$ is a constant independent of $s$.
\end{lem}
\begin{proof}
(i) We first prove that $a(\cdot,\,\cdot)$ is continuous and strictly coercive.
Using the Cauchy-Schwarz inequality, the boundedness of  $\mathscr{G}$ in Lemma \ref{lem:bd} and the trace theorem, we obtain
\ben
 |a(u_L,\, v) | &\leq &|s|^{-1}\|\nabla u_L\|_{L^2(\Omega_R^+)} \|\nabla v\|_{L^2(\Omega_R^+)}+ |s|\|u_L\|_{L^2(\Omega_R^+)} \|v\|_{L^2(\Omega_R^+)} \\
& &+ |s|^{-1}\|\mathscr{G} u_L\|_{H^{-1/2}(\Gamma_R^+)} \|v\|_{H^{1/2}(\Gamma_R^+)}\\
&\leq&  C \|u_L\|_{H^1(\Omega_R^+)} \|v\|_{H^1(\Omega_R^+)} +  C \|u_L\|_{H^1(\Omega_R^+)} \|v\|_{H^1(\Omega_R^+)}\\
&& C \| u_L\|_{H^{1/2}(\Gamma_R^+)} \| v\|_{H^{1/2}(\Gamma_R^+)} \\
&\leq&  C \|u_L\|_{H^1(\Omega_R^+)} \|v\|_{H^1(\Omega_R^+)}.
\enn
Setting $v=u_L$,  it follows from the expression of the sesquilinear form $a(\cdot,\cdot)$ that
\ben
a(u_L,\,u_L)=\int_{\Om_R^+} \frac{1}{s} |\nabla u_L|^2 +s |u_L|^2  \,d x -\langle s^{-1}\mathscr{G}u_L, \, u_L \rangle_{\Gamma_R^+}.
\enn
Taking  the real part of  the above equation and  using Lemma \ref{lem:G}   we have
\be \label{eq:s-cor}
\textnormal{Re}(a(u_L,\,u_L))&=&\int_{\Om_R^+} \frac{s_1}{|s|^2} |\nabla u_L|^2 +s_1 |u_L|^2  \,d x -\textnormal{Re} \langle s^{-1}\mathscr{G}u_L, \, u_L \rangle_{\Gamma_R^+}\nonumber\\
&\geq&  \int_{\Om_R^+} \frac{s_1}{|s|^2} |\nabla u_L|^2 +s_1 |u_L|^2  \,d x \nonumber \\
&\geq&   \frac{s_1}{|s|^2} \left(  \|\nabla u_L\|_{L^2(\Omega_R^+)}^2 + \|s u_L\|_{L^2(\Omega_R^+)}^2 \right ).
\en
Hence, by the Lax-Milgram Lemma, the variational problem (\ref{eq:s-v}) has a unique solution $u_L\in X_R$.

(ii) Combining (\ref{eq:s-v}) with the Cauchy-Schwarz inequality, it follows that
\be \label{eq:s-cs}
|a(u_L,\,u_L)|
&\leq& \frac{1}{|s|} \| f_L\|_{L^2(\Om_R^+)} \| s u_L\|_{L^2(\Om_R^+)} \nonumber\\
&\leq& \frac{1}{|s|} \| f_L\|_{L^2(\Om_R^+)} \| s u_L\|_{H^1(\Om_R^+)}\nonumber\\
&\leq& \frac{1}{|s|} \| f_L\|_{L^2(\Om_R^+)} \left(|s|^2\| \nabla u_L\|_{L^2(\Om_R^+)}^2 +\| s u_L\|_{L^2(\Om_R^+)}^2\right)^{1/2}\nonumber\\
&\leq& C \frac{1+|s|}{|s|} \| f_L\|_{L^2(\Om_R^+)} \left(\| \nabla u_L\|_{L^2(\Om_R^+)}^2 +\| s u_L\|_{L^2(\Om_R^+)}^2\right)^{1/2}.
\en
Combining (\ref{eq:s-cor}) and (\ref{eq:s-cs}) yields
\ben
&&\frac{s_1}{|s|^2} \left(  \|\nabla u_L\|_{L^2(\Omega_R^+)}^2 + \|s u_L\|_{L^2(\Omega_R^+)}^2 \right ) \\
&\leq&   \textnormal{Re}(a(u_L,\,u_L)) \\
&\leq&  |a(u_L,\,u_L) | \\
&\leq&  C \frac{1+|s|}{|s|} \| f_L\|_{L^2(\Om_R^+)} \left(\| \nabla u_L\|_{L^2(\Om_R^+)}^2 +\| s u_L\|_{L^2(\Om_R^+)}^2\right)^{1/2}.
\enn
Then, using the Cauchy-Schwarz inequality, we have
\ben
 \|\nabla u_L\|_{L^2(\Omega_R^+)} + \|s u_L\|_{L^2(\Omega_R^+)} &\leq&  \left(\| \nabla u_L\|_{L^2(\Om_R^+)}^2 +\| s u_L\|_{L^2(\Om_R^+)}^2\right)^{1/2} \\
 &\leq& C \frac{(1+|s|)|s|}{s_1}\| f_L\|_{L^2(\Om_R^+)},
\enn
which completes the proof of the stability estimate.
\end{proof}

\begin{thm}
The initial boundary value problem (\ref{eqs:wave-b}) has a unique solution
\ben
u(x,t)\in L^2(0,T; X_R) \cap H^1(0,T; L^2(\Om_R^+)),
\enn
which satisfies the stability estimate
\ben
\max_{0\leq t\leq T}\Big(\| \pa_t u \|_{L^2(\Omega_R^+)} + \| \nabla u \|_{L^2(\Omega_R^+)} \Big ) \leq C  \|\pa_t f\|_{L^1(0,T;L^2(\Omega_R^+))}.
\enn
\end{thm}
\begin{proof}
We first prove existence and uniqueness of solutions to (\ref{eqs:wave-b}).
Simple calculations show that
\ben
&&\int_0^T \| \pa_t u \|_{L^2(\Omega_R^+)}^2 + \| \nabla u \|_{L^2(\Omega_R^+)}^2 \, dt\\
\leq && C \int_0^\infty   e^{-2s_1 t} \Big(\| \pa_t u \|_{L^2(\Omega_R^+)}^2 + \| \nabla u \|_{L^2(\Omega_R^+)}^2\Big) \, dt.
\enn
Hence it suffices to estimate the integral
\ben
 \int_0^\infty   e^{-2s_1 t} \Big(\| \pa_t u \|_{L^2(\Omega_R^+)}^2 + \| \nabla u \|_{L^2(\Omega_R^+)}^2\Big) \, dt.
\enn

Based on the stability estimate of $u_L$ in Lemma \ref{lem:wave-s}, we derive from \cite[Lemma 44.1]{Treves} that $u_L$ is a holomorphic function of $s$ on the half plane $s_1>\zeta_0>0$, where $\zeta_0$ is any positive constant. Thus, by Lemma \ref{lem:a},   the inverse Laplace transform of  $u_L$ exists and is supported in $[0,\infty)$.

Set $u=\mathscr{L}^{-1}(u_L)$.  Applying the Parseval identity (\ref{PI}) and the stability estimate (\ref{eq:wave-s}) in Lemma \ref{lem:wave-s} and  using the Cauchy-Schwarz inequality, we obtain
\ben
 &&\int_0^\infty   e^{-2s_1 t} (\| \pa_t u \|_{L^2(\Omega_R^+)}^2 + \| \nabla u \|_{L^2(\Omega_R^+)}^2) \, dt\\
 =&& \frac{1}{2\pi} \int_{-\infty}^{+\infty}  \| s  u_L \|_{L^2(\Omega_R^+}^2 + \| \nabla u_L \|_{L^2(\Omega_R^+)}^2) \, d s_2\\
 \leq && C \frac{1}{s_1^2} \int_{-\infty}^{+\infty} (1+|s|)^2|s|^2 \| f_L\|_{L^2(\Omega_R^+)}^2 \, d s_2 \\
 \leq && C \frac{1}{s_1^2} \int_{-\infty}^{+\infty} \| s^2f_L\|_{L^2(\Omega_R^+)}^2 +\|sf_L\|_{L^2(\Omega_R^+)}^2\, d s_2 \\
 \leq && C \frac{1}{s_1^2}  \int_0^\infty e^{-2 s_1 t} \left(\|\pa_t^2 f\|_{L^2(\Omega_R^+)}^2+\|\pa_t f\|_{L^2(\Omega_R^+)}^2\right) \, dt.
\enn
This together with the Poincar\'{e} inequality proves
\ben
u\in L^2(0,T; X_R) \cap H^1(0,T; L^2(\Om_R^+)).
\enn
To prove the stability estimate, we define the energy function
\ben
E(t):=\frac{1}{2}\left(\| \pa_t u \|_{L^2(\Omega_R^+)}^2 + \| \nabla u \|_{L^2(\Omega_R^+)}^2\right), \quad 0<t<T.
\enn
It is obvious that
\ben  \label{eq:en}
E(t)-E(0)=\int_0^t E'(\tau)\,d\tau.
\enn
 Recalling the wave equation in (\ref{eqs:wave-b}) and applying integration by parts, we obtain
\be \label{eq:eni}
\int_0^t e^{-2s_1 \tau} E'(\tau)\,d\tau\quad =&&\mbox{Re}\int_0^t e^{-2s_1 \tau}\int_{\Om_R^+} \pa_{\tau\tau}u \pa_{\tau}\ov{u} + \nabla u\cdot \nabla (\pa_{\tau}\ov{u}) \, dx d\tau \nonumber\\
=&&\mbox{Re}\int_0^t e^{-2s_1 \tau} \int_{\Om_R^+} (\Delta u+\pa_\tau f) \pa_{\tau}\ov{u} + \nabla u\cdot \nabla (\pa_{\tau}\ov{u}) \, dx d\tau \nonumber \\
=&&\mbox{Re}\int_0^t e^{-2s_1 \tau}\langle \mathscr{T}u,\, \pa_{\tau}\ov{u} \rangle_{\Gamma_R^+}  \,d\tau \nonumber  \\&&
+\mbox{Re}\int_0^t e^{-2s_1 \tau}(\pa_\tau f,\, \pa_{\tau}\ov{u})_{L^2(\Om_R^+)}\,d\tau.
\en
Applying integration by parts on the left hand side of \eqref{eq:eni} and using $E(0)=0$ together with Lemma \ref{lem:t}, we obtain
\ben
  &&e^{-2s_1t}E(t)+2s_1\int_0^t e^{-2s_1\tau}E(\tau)\,d\tau \\
\leq && \mbox{Re}\int_0^t e^{-2s_1\tau} (\pa_\tau f,\, \pa_{\tau}\ov{u})_{L^2(\Om_R^+)}\,d\tau\\
\leq && \int_0^T \|e^{-s_1t}\pa_t f\|_{L^2(\Om_R^+)}\,\|e^{-s_1t}\pa_t u\|_{L^2(\Om_R^+)} dt\\
\leq && \max_{0\leq t\leq T} \| e^{-s_1t}\pa_t u \|_{L^2(\Omega_R^+)}\,  \|e^{-s_1t}\pa_t f\|_{L^1(0,T;L^2(\Omega_R^+))}\\
\leq && \vep \max_{0\leq t\leq T} \| e^{-s_1t}\pa_t u \|_{L^2(\Omega_R^+)}^2 + \frac{1}{4\vep}\,  \|e^{-s_1t}\pa_t f\|_{L^1(0,T;L^2(\Omega_R^+))}^2.
\enn
Letting $s_1\rightarrow0$, choosing $\vep>0$ small enough and applying Cauchy-Schwartz inequality, we finally get
\ben
&&\max_{0\leq t\leq T} \Big(\| \pa_t u \|_{L^2(\Omega_R^+)} + \| \nabla u \|_{L^2(\Omega_R^+)} \Big)\\
\leq &&C \max_{0\leq t\leq T} \Big(\| \pa_t u \|_{L^2(\Omega_R^+)}^2 + \| \nabla u \|_{L^2(\Omega_R^+)}^2 \Big)^{1/2}\\
\leq &&C  \|\pa_t f\|_{L^1(0,T;L^2(\Omega_R^+))}.
\enn
This completes the stability estimate.
\end{proof}

\section{The time-domain PML problem }
Inspired by the PML approach for bounded obstacles \cite{Chen09,Chen12},
we present in this section the time-domain PML formulation in a perturbed half-plane and then
show the well-posedness and stability of the PML problem by applying the Laplace transform together with the variational and energy methods.

\subsection{Well-posedness of the PML problem}

We surround the domain $\Om_R^+$ with a PML layer $$\Om_{PML}^+:=B_{\rho}^+\ba \ov{B_R^+}= :\{x\in\Omega:R<|x|<\rho\},$$ where $B_{\rho}^+:=\{x\in\Omega:|x|<\rho\}$. We denote $\Om_{\rho}^+:=B_{\rho}^+\ba \ov{D}$ the truncated PML domain with the exterior boundary $\Gamma_{\rho}^+:=\{x\in\Omega:|x|=\rho\}$. Let $s_1=\textnormal{Re}(s)>0$ for $s\in\C$. Define the medium parameter in the PML layer as
\ben
\alpha(r)= \left\{\begin{array}{lll}
1,  && r\leq R, \\
1+s^{-1} \sigma(r),  && r> R,
\end{array}\right.
\enn
where $r=|x|$,  $\sigma=0$ for $r \leq R$ and $\sigma >0$ for $r> R$.

In what follows, we will derive the PML formulation by a complex transformation of variables.  Denote by $\tilde{r}$ the complex radius
\ben
\tilde{r}=\int_0^r \alpha(\tau)\,d\tau=r \beta(r),
\enn
where $\beta(r)=r^{-1}\int_0^r \alpha(\tau)\,d\tau$. It is obvious that $\beta(r)=1+s^{-1}\hat{\sigma}(r)$ for $r\geq R$, where $\hat{\sigma}(r)=r^{-1}\int_0^r \sigma(\tau)\,d\tau$.
To derive the PML equations, we need to transform the exterior problem (\ref{eqs:wave-b}) into the  s-domain. On $\Gamma_R^+$, the Laplace transform of $u$ can be expanded into the series,
\ben
u_L (R,s)=\sum_{n=1}^{\infty} u^n _L(R,s)\sin{n\theta},\quad u^n _L(R,s)=\frac{2}{\pi} \int_{0}^{\pi} {u_L (R, \theta, s) \sin{n\theta} \, d\theta}.
\enn
Then, let us define the PML extension $\tilde{u}_L$ in the s-domain as
\ben
\tilde{u}_L( r,\theta,s) =\sum_{n=1}^{\infty} \frac{K_n(s\tilde r)}{K_n(sR)}\tilde u^n _L(R,s)\sin{n\theta}, \quad r>R.
\enn
Since $K_n(z)\backsim (\frac{\pi}{2z})^{1/2}e^{-z}$ as $|z|\rightarrow\infty$,  $\tilde{u}_L(r, \theta, s)$ decays exponentially for large $\tilde r$. It is easy to see that $\tilde u _L$ satisfies $-\frac{1}{\tilde r} \frac{\pa}{\pa \tilde r}(\tilde r \frac{\pa}{\pa \tilde r})\tilde{u}_L-\frac{1}{\tilde r ^2}\frac{\pa^2}{\pa \theta^2}\tilde{u}_L+s^2\tilde{u}_L=0$ in $\Omega \backslash \overline{B_R^+}$.
Since $\tilde{r}=r\beta$ and $d\tilde{r}=\alpha dr$, we obtain
\be \label{eq:s-pml}
-\nabla\cdot ( A \nabla \tilde u_L)+s^2\alpha \beta \tilde u_L=0, \quad x\in \Omega \backslash \overline{B_R^+}
\en
where $A=\textnormal{diag}\{\beta / \alpha, \alpha /  \beta\}$ is a complex matrix and $ A \nabla \tilde u_L=\frac{\beta}{\alpha} \frac{\pa \tilde u_L}{\pa r} \textbf{e}_r+\frac{\alpha}{\beta r}\frac{\pa \tilde u_L}{\pa \theta}\textbf{e}_\theta$. Here $\textbf{e}_r$ and $\textbf{e}_\theta$ are the unit vectors in polar coordinates.


Next, we will deduce the PML system in the time-domain by applying the inverse Laplace transform to (\ref{eq:s-pml}). Since $A$, $\alpha$ and $\beta$ are complex,  to simplify the inverse Laplace transform, we introduce the auxiliary functions  
\be \label{eq:s-m}
\tilde{p}_L^{*}:=-\frac{1}{s}\nabla \tilde u_L, \quad \tilde u _L^{*}:=\frac{1}{s}\sigma \tilde u_L,\quad  \tilde p_L:=A  \tilde{p}_L^{*},
\en
to transform (\ref{eq:s-pml}) into a first order system.

In $\Om\ba \ov {B_{R}^{+}}$, define
\ben
\tilde u:=\mathscr{L}^{-1}(\tilde u_L),\;\tilde p:=\mathscr{L}^{-1}(\tilde p_L),\; \tilde u^*:=\mathscr{L}^{-1}(\tilde u_L^*),\;\tilde p^*:=\mathscr{L}^{-1}(\tilde p_L^*),
\enn
with the zero initial conditions
\ben
\tilde u|_{t=0}=0,\;\tilde p|_{t=0}=0,\;\tilde u^*|_{t=0}=0,\;\tilde p^*|_{t=0}=0.
\enn
Taking the inverse Laplace transform to (\ref{eq:s-pml}) and (\ref{eq:s-m}) and using the zero initial conditions, we can write the PML system for $x\in \Om \ba \ov {B_R^+}$ as
\be \label{t-pml}
\left\{\begin{array}{lll}
\pa_t \tilde u+(\sigma +\hat \sigma)\tilde u +\sigma \tilde u^{*}+ \nabla \cdot \tilde p =0,\\
\pa_t \tilde{p}^{*}=-\nabla\tilde u,\quad \pa_t \tilde{u}^{*}=\sigma \tilde u, \\
\pa_t \tilde{p}+\Lambda_1 \tilde p= \pa_t \tilde{p}^{*}+\Lambda_2 \tilde p^{*},
\end{array}\right.
\en
where $s\alpha=s+\sigma$,  $s\beta=s+\hat\sigma$, $\Lambda_1={M^{T}}\textnormal{diag}\{\sigma, \hat\sigma\} {M}$ and $\Lambda_2={M^{T}}\textnormal{diag}\{\hat\sigma, \sigma\} {M}$ with
\ben
M := \left ( \begin{array}{rrl}
  \cos \theta  && \sin \theta \\
-\sin \theta  && \cos \theta
\end{array}\right).
\enn

Since the above PML system (\ref{t-pml}) is a  first order system,  it is necessary to reduce  equivalently  the time-domain scattering problem (\ref{eqs:wave-b}) in the half space into a first order PDE system:
\be  \label{eqs:h-1}
\left\{\begin{array}{lll}
\partial_t u =-\nabla\cdot p+f(x,t) &&\quad \mbox{in} \quad  \Omega_R^+\times(0, T),\\
\partial_t p=-\nabla u  &&\quad \mbox{in} \quad  \Omega_R^+\times(0, T),\\
u=0  &&\quad \mbox{on}\quad  (\partial D \cup \Gamma_0) \times(0,T),\\
p\cdot\hat{x}+\mathscr{T}(\int_0^t u\,d\tau)=0, &&\quad \mbox{on}\quad \Gamma_R^+\times (0,T), \\
u|_{t=0}=p|_{t=0}=0 &&\quad \mbox{in} \quad  \Omega_R^+.
       \end{array}\right.
\en
 Below we derive the DtN boundary condition on $\Gamma_R^{+}\times (0,T)$. Taking Laplace transform to the second equation of (\ref{eqs:h-1}), we obtain that
\ben
p_L+\frac{1}{s}\nabla u_L=0.
\enn
Then, multiplying the above equation by $\hat{x}=x/|x|$ on $\Gamma_R^{+}$ and using the DtN boundary condition $\pa_r u_L=\mathscr{G}u_L$, it follows that
\be \label{bc:dtn-pg}
p_L \cdot \hat x+ \frac{1}{s}\mathscr{G}u_L=0 \quad \mbox{on}\quad \Gamma_R^{+}.
\en
Taking inverse Laplace transform to  (\ref{bc:dtn-pg}) and using (\ref{eq:l-3}), we have
\be\label{bc:dtn-1}
p\cdot\hat{x}+\mathscr{T}\left(\int_0^t u\,d\tau\right)=0 \quad \mbox{on}\quad \Gamma_R^{+}\times(0,T).
\en

Further, since $\sigma(R)=\hat\sigma(R)=0$, we get $\alpha =\beta=1$ on $\Gamma_R^+$ and thus $\tilde u=u$ and $\tilde p=p$ on $\Gamma_R^+$. Therefore, ($\tilde u, \tilde p $) can be viewed as the extension of the solution of the problem (\ref{eqs:wave}). Setting $\tilde u=u $ and $\tilde p=p$ in $\Om_R^+$, we can reformulate the truncated PML problem in $\Om_{\rho}^{+}$ as
\begin{subequations} \label{eq:t-pml}
\begin{align}
&\pa_t \tilde u+(\sigma +\hat \sigma)\tilde u +\sigma \tilde u^{*}+ \nabla \cdot \tilde p =f \quad &&\mbox{in}\quad \Om_\rho^+\times (0,T) \label{eq:t-pml-a},\\
&\pa_t \tilde{p}^{*}=-\nabla\tilde u,\quad \pa_t \tilde{u}^{*}=\sigma \tilde u  \quad &&\mbox{in}\quad \Om_\rho^+\times (0,T) \label{eq:t-pml-b},\\
&\pa_t \tilde{p}+\Lambda_1 \tilde p= \pa_t \tilde{p}^{*}+\Lambda_2 \tilde p^{*} \quad &&\mbox{in}\quad \Om_\rho^+\times (0,T) \label{eq:t-pml-c},\\
&\tilde u =0 \quad &&\mbox{on}\quad (\pa D \cup \Gamma_0) \times (0,T)\label{eq:t-pml-d},\\
&\tilde u =0 \quad &&\mbox{on}\quad  \Gamma_\rho^+ \times (0,T)\label{eq:t-pml-e},\\
&\tilde u|_{t=0} =\tilde p|_{t=0} =\tilde u^{*}|_{t=0}=\tilde p^{*}|_{t=0} \quad &&\mbox{in}\quad  \Om_\rho^+. \label{eq:t-pml-f}
\end{align}
\end{subequations}

The well-posedness of truncated PML problem will be proved by applying Laplace transform and the variational method. In the rest of this paper, we assume that $\sigma(r) $ is monotonically increasing on $[R, \rho]$ such that $\sigma_R \leq \sigma \leq \sigma_\rho$.  First, we take Laplace transform to (\ref{eq:t-pml}) with $s\in\C$ and then eliminate $\tilde p_L$, $\tilde u_L^{*}$ and $\tilde p_L^{*}$, to obtain
\be \label{eqs:s-pml}
\left\{\begin{array}{lll}
-\nabla\cdot ( A  \nabla \tilde u_L)+s^2\alpha \beta \tilde u_L=sf_L  \quad &&\mbox{in}\quad \Om_\rho^+ \times (0,T),\\
\tilde u_L =0 \quad &&\mbox{on}\quad \pa D \cup \Gamma_0,\\
\tilde u_L =0 \quad &&\mbox{on}\quad  \Gamma_\rho^+.
\end{array}\right.
\en
It is easy to derive the variational formulation of (\ref{eqs:s-pml}): find a solution $\tilde u_L\in H_0^1(\Om_\rho^+)$ such that
\be\label{eq:v-pml}
\tilde a(\tilde u_L, v)=\int_{\Om_\rho^+} s f_L \ov v dx, \quad \mbox{for all}\; v\in H_0^1(\Om_\rho^+ )
\en
where the sesquilinear form $\tilde a(\cdot,\, \cdot):H_0^1(\Om_\rho^+ )\times H_0^1(\Om_\rho^+ )\rightarrow \C$ is defined as
\ben
\tilde a(\tilde u_L, v)=\int_{\Om_\rho^+} A \nabla \tilde u_L \cdot \nabla \ov v+s^2 \alpha \beta \tilde u_L \ov v \,dx.
\enn
We will prove the well-posedness of (\ref{eqs:s-pml}). The proof of the  first inequality in the subsequent lemma is similar to that  in \cite[Lemma 4.1]{Chen09} where the PML layer is defined as an annular domain in the free space $\R ^2$  and  $\sigma$ is  a positive constant.

\begin{lem}\label{lem:a-pml}
For any $\tilde u_L\in H_0^1(\Om_\rho ^+ )$, it holds that
\begin{itemize}
\item[(a)] $\textnormal{Re}[\tilde a(u_L,u_L)]+\frac{s_2}{s_1+\sigma_\rho}\textnormal{Im}[\tilde a(u_L,u_L)]\geq \frac{s_1^2}{(s_1+\sigma_\rho)^2}\big(\|A\nabla u_L\|_{L^2(\Om_\rho^+)}^2+\|s\alpha \beta u_L\|_{L^2(\Om_\rho^+)}^2\big)$,\\
\item[(b)] $|\tilde a(u_L,u_L)|\geq \Big( \frac{s_1}{s_1+\sigma_\rho}\Big)^2\frac{s_1}{|s|}|\frac{s_1}{s+\sigma_\rho}|^2\big(\|\nabla u_L\|_{L^2(\Om_\rho^+)}^2+\|s u_L\|_{L^2(\Om_\rho^+)}^2\big)$.
\end{itemize}
\end{lem}
\begin{proof}
It suffices to prove (b).  For any $\tilde u_L\in H_0^1(\Om_\rho ^+ )$, applying (a) we have
\ben
|\tilde a(u_L,u_L)|&\geq &\frac{1}{|s|}\mbox{Re}[s \tilde a(u_L,u_L)] \\
&\geq &\frac{s_1}{|s|}\left(\textnormal{Re}[\tilde a(u_L,u_L)]+\frac{s_2}{s_1}\textnormal{Im}[\tilde a(u_L,u_L)]\right)\\
&\geq &\frac{s_1}{|s|}\left(\textnormal{Re}[\tilde a(u_L,u_L)]+\frac{s_2}{s_1+\sigma_\rho}\textnormal{Im}[\tilde a(u_L,u_L)]\right)\\
&\geq &\frac{s_1}{|s|}\left(\frac{s_1}{s_1+\sigma_\rho}\right)^2\left(\|A\nabla u_L\|_{L^2(\Om_\rho^+)}^2+\|s\alpha \beta u_L\|_{L^2(\Om_\rho^+)}^2\right)\\
&\geq &\Big(\frac{s_1}{s_1+\sigma_\rho}\Big)^2\frac{s_1}{|s|}\Big|\frac{s_1}{s+\sigma_\rho}\Big|^2\left(\|\nabla u_L\|_{L^2(\Om_\rho^+)}^2+\|s u_L\|_{L^2(\Om_\rho^+)}^2\right).
\enn
This completes the proof.
\end{proof}

\begin{lem} \label{lem:e-pml}
The variational problem (\ref{eq:v-pml}) has a unique solution $\tilde u_L\in H_0^1(\Om_R^+)$ with the  following stability estimates
\be
\|A\nabla u_L\|_{L^2(\Om_\rho^+)}+\|s\alpha \beta u_L\|_{L^2(\Om_\rho^+)}&\leq & C \left(\frac{|s|}{s_1}\right)^{1/2}\left(1+\frac{\sigma_\rho}{s_1} \right)\| f_L \|_{L^2(\Om_\rho^+)}, \label{eq:s-e1}\\
\|\nabla u_L\|_{L^2(\Om_\rho^+)}+\|s u_L\|_{L^2(\Om_\rho^+)}
&\leq & C \left(\frac{|s|}{s_1}\right)^{1/2}\left(1+\frac{\sigma_\rho}{s_1} \right) \frac{|s+\sigma_\rho|}{s_1} \| f_L \|_{L^2(\Om_\rho^+)},\label{eq:s-e2}
\en
where $C$ is a constant independent of $s$.
\end{lem}

\begin{proof}
The first part of the lemma follows easily from the Lax-Milgram lemma and the strictly coercivity of $\tilde a(\cdot,\,\cdot)$ in Lemma \ref{lem:a-pml}. Further,  the stability estimates (\ref{eq:s-e1}) and (\ref{eq:s-e2}) follow from (\ref{eq:v-pml}), Lemma \ref{lem:a-pml} and the Cauchy-Schwartz inequality.
\end{proof}

The well-posedness of PML problem (\ref{eq:t-pml}) in the time domain can be established by applying Lemma \ref{lem:e-pml}.
\begin{thm}
The truncated PML problem (\ref{eq:t-pml}) in the time domain has a unique solution $(u,p,u^*,p^*)$ such that
\ben
&u\in L^2(0,T;H_0^1(\Omega_\rho^+))  \cap  H^1(0,T;L^2(\Omega_\rho^+)),\quad & u^*\in H^1(0,T;L^2(\Omega_\rho^+)),\\
&p\in L^2(0,T;H(\textnormal{div},\Omega_\rho^+))  \cap
H^1(0,T;L^2(\Omega_\rho^+)),\quad  & p^*\in H^1(0,T;L^2(\Omega_\rho^+)).
\enn
\end{thm}
\begin{proof}
By simple calculations, we can obtain
\ben
&&\int_0^T \| \pa_t u \|_{L^2(\Om_\rho^+)}^2 + \| \nabla u \|_{L^2(\Om_\rho^+)}^2 \, dt\\
\leq && C \int_0^\infty   e^{-2s_1 t} (\| \pa_t u \|_{L^2(\Om_\rho^+)}^2 + \| \nabla u \|_{L^2(\Om_\rho^+)}^2) \, dt.
\enn
Hence it suffices to estimate the integral
\ben
 \int_0^\infty   e^{-2s_1 t} (\| \pa_t u \|_{L^2\Om_\rho^+)}^2 + \| \nabla u \|_{L^2(\Om_\rho^+)}^2) \, dt.
\enn
Using the stability estimate of $u_L$ in Lemma \ref{lem:e-pml}, we duduce from \cite[Lemma 44.1]{Treves} that $u_L$ is a holomorphic function of $s$ on the half plane $s_1>\zeta_0>0$, where $\zeta_0$ is any positive constant. Thus, by Lemma \ref{lem:a}, the inverse Laplace transform of  $u_L$ exists and is supported in $[0,\infty]$.

Set $u=\mathscr{L}^{-1}(u_L)$. One deduces from the Parseval identity (\ref{PI}),  stability estimate (\ref{eq:s-e2}) and the Cauchy-Schwartz inequality that
\ben
 &&\int_0^\infty   e^{-2s_1 t} (\| \pa_t u \|_{L^2(\Om_\rho^+)}^2 + \| \nabla u \|_{L^2(\Om_\rho^+)}^2) \, dt\\
 =&& \frac{1}{2\pi} \int_{-\infty}^{+\infty}  \| s  u_L \|_{L^2(\Om_\rho^+)}^2 + \| \nabla u_L \|_{L^2(\Om_\rho^+)}^2) \, d s_2\\
 \leq && C \frac{1}{s_1^3}\left( 1+\frac{\sigma_\rho}{s_1}\right)^2 \int_{-\infty}^{+\infty} |s| |s+\sigma_\rho|^2  \| f_L\|_{L^2(\Om_\rho^+)}^2 \, d s_2 \\
 = && C \frac{1}{s_1^3}\left( 1+\frac{\sigma_\rho}{s_1}\right)^2 \int_{-\infty}^{+\infty}   \|s(s+\sigma_\rho)f_L\|_{L^2(\Om_\rho^+)}   \| (s+\sigma_\rho) f_L\|_{L^2(\Om_\rho^+)}\, d s_2 \\
\leq && C \frac{1}{s_1^3}\left( 1+\frac{\sigma_\rho}{s_1}\right)^2 \int_{0}^{+\infty} e^{-2 s_1 t}  \|\pa_{tt}f+\sigma_\rho\pa_t f\|_{L^2(\Om_\rho^+)}   \|  \pa_t f+\sigma_\rho f\|_{L^2(\Om_\rho^+)}\, d t \\
\leq && C \frac{1}{s_1^3}\left( 1+\frac{\sigma_\rho}{s_1}\right)^2 \int_{0}^{+\infty} e^{-2 s_1 t} \Big( \|\pa_{tt}f\|_{L^2(\Om_\rho^+)} \|\pa_t f\|_{L^2(\Om_\rho^+)}  \\
 &&+\sigma_\rho\|\pa_{tt}f\|_{L^2(\Om_\rho^+)}   \|  f\|_{L^2(\Om_\rho^+)}  +\sigma_\rho\|\pa_t f\|_{L^2(\Om_\rho^+)}^2
 +\sigma_\rho^2\|\pa_t f\|_{L^2(\Om_\rho^+)}   \|  f\|_{L^2(\Om_\rho^+)}  \Big) \, d t \\
 \leq && C \frac{1}{s_1^3}\left( 1+\frac{\sigma_\rho}{s_1}\right)^2 \int_{0}^{+\infty} e^{-2 s_1 t} \Big( (1+\sigma_\rho)\|\pa_{tt}f\|_{L^2(\Om_\rho^+)}^2 +\sigma_\rho(1+\sigma_\rho+\sigma_\rho^2)\|\pa_t f\|_{L^2(\Om_\rho^+)}^2 \\
 &&+ (1+\sigma_\rho)\|  f\|_{L^2(\Om_\rho^+)}^2  \Big)\, dt.
\enn
This together with the Poincar\'{e} inequality proves
\ben
u\in L^2(0,T;H_0^1(\Om_\rho^+)) \cap H^1(0,T; L^2(\Om_\rho^+)).
\enn
From (\ref{eq:s-m}) and the first equation of (\ref{eqs:s-pml}), we  obtain
\be \label{eq:t-tr}
s p_L=-A \nabla u_L, \quad \nabla \cdot p_L=-s \alpha\beta u_L+f_L.
\en
By the first equation of (\ref{eq:t-tr}) and stability estimate (\ref{eq:s-e1}), we deduce from \cite[Lemma 44.1]{Treves} that $p_L$ is holomorphic function of $s$ on the half plane $s_1>\zeta_0>0$, where $\zeta_0$ is any positive constant. Thus, by Lemma \ref{lem:a}, it follows from that  the inverse Laplace transform of  $p_L$ exists and is supported in $[0,\infty]$.
Then, using the Parseval identity (\ref{PI}), Cauchy inequality with $\vep$  and  stability estimate (\ref{eq:s-e1}), we can obtain
\ben
 &&\int_0^\infty   e^{-2s_1 t} \left(\| \pa_t p \|_{L^2(\Om_\rho^+)}^2 + \| \nabla \cdot p \|_{L^2(\Om_\rho^+)}^2 \right) \, dt\\
 = && \frac{1}{2\pi} \int_{-\infty}^{+\infty}  \| s  p_L \|_{L^2(\Om_\rho^+)}^2 + \| \nabla \cdot  p_L \|_{L^2(\Om_\rho^+)}^2 \, d s_2\\
  =&& \frac{1}{2\pi} \int_{-\infty}^{+\infty}  \| A \nabla u_L \|_{L^2(\Om_\rho^+)}^2 + \| s \alpha \beta   u_L +f_L\|_{L^2(\Om_\rho^+)}^2 \, d s_2\\
 \leq && C \frac{1}{2\pi} \int_{-\infty}^{+\infty}  \| A \nabla u_L \|_{L^2(\Om_\rho^+)}^2 + \| s \alpha \beta   u_L\|_{L^2(\Om_\rho^+)}^2+\| f_L\|_{L^2(\Om_\rho^+)}^2 \, d s_2\\
 \leq &&  C \frac{1}{2\pi} \int_{-\infty}^{+\infty}  \frac{|s|}{s_1} \left(1+\frac{\sigma_\rho}{s_1} \right)^2\| f_L \|_{L^2(\Om_\rho^+)}^2+ \| f_L \|_{L^2(\Om_\rho^+)}^2 \,d s_2\\
  \leq &&  C  \int_{0}^{+\infty} e^{-2s_1 t} \left( \frac{1}{s_1} \left(1+\frac{\sigma_\rho}{s_1} \right)^2 \|\pa_t f \|_{L^2(\Om_\rho^+)}\| f \|_{L^2(\Om_\rho^+)}+ \| f \|_{L^2(\Om_\rho^+)}^2 \right) \,d t\\
   \leq &&  C  \int_{0}^{+\infty} e^{-2s_1 t} \left( \frac{1}{s_1} \left(1+\frac{\sigma_\rho}{s_1} \right)^2 \|\pa_t f \|_{L^2(\Om_\rho^+)}^2+ \Big( \frac{1}{s_1} \left(1+\frac{\sigma_\rho}{s_1} \right)^2+1\Big)\| f \|_{L^2(\Om_\rho^+)}^2 \right) \,d t.
\enn
Hence,
\ben
p\;\in\; L^2(0,T;H(\textnormal{div},\Omega_\rho^+))  \cap
H^1(0,T;L^2(\Omega_\rho^+)).
\enn
By the second equation of  (\ref{eq:t-pml-b}) and Poinc$\acute{a}$re's inequality, we see
\ben
 &&\int_0^\infty   e^{-2s_1 t} \| \pa_t u^* \|_{L^2(\Om_\rho^+)}^2  \, dt\\
\leq &&\int_0^\infty   e^{-2s_1 t} \sigma_\rho^2 \| u\|_{L^2(\Om_\rho^+)}^2  \, dt\\
 \leq && C \int_0^\infty   e^{-2s_1 t} \sigma_\rho^2 \left(\| \nabla u\|_{L^2(\Om_\rho^+)}^2  \right) \, dt\\
 \leq && C \int_0^\infty   e^{-2s_1 t} \sigma_\rho^2 \left(\| \nabla u\|_{L^2(\Om_\rho^+)}^2 + \| \pa_t u\|_{L^2(\Om_\rho^+)}^2  \right) \, dt.
\enn
Then, in view of the  solution space for $u$, we know $u^*\;\in\; H^1(0,T;L^2(\Omega_\rho^+))$.
Similarly, from the first equation of (\ref{eq:t-pml-b}), we know $\| \pa_t p^*\|_{L^2(\Om_\rho^+)}^2=\| \nabla u\|_{L^2(\Om_\rho^+)}^2$, which implies $p^*\;\in\;H^1(0,T;L^2(\Omega_\rho^+))$.
\end{proof}

\subsection{Stability of the truncated PML problem }
The aim of this subsection is to prove the stability of the PML problem (\ref{eq:t-pml}) with $\sigma=\hat \sigma$. We first present an auxiliary stability estimate.
\begin{thm} \label{thm:sg}
Let $(u,p,u^*,p^*)$ be the solution of the truncated PML problem (\ref{eq:t-pml}). Then there holds the stability estimate
\ben
&&\max_{0\leq t\leq T} \left( \|\pa_t u +\sigma u \|_{L^2(\Omega_\rho^+)} +\|\pa_t p +\sigma p \|_{L^2(\Omega_\rho^+)} +\|\pa_t u^* +\sigma u^* \|_{L^2(\Omega_\rho^+)} \right)\\
\leq && C \int_0^T \| \pa_t f +\sigma f\|_{L^2(\Omega_\rho^+)}\, dt,
\enn
where the constant $C$ is independent of $\sigma$ and $T$.
\end{thm}
\begin{proof}
We apply to  equation  (\ref{eq:t-pml-a}) the operator $\pa_t +\sigma$ to get
\ben
\pa_{t}^2 u+\nabla \cdot \left( \pa_t p +\sigma p \right) + (\sigma+\hat \sigma)\left( \pa_t u +\sigma u \right) +\hat \sigma \left( \pa_t u^* +\sigma u^* \right)= \pa_t f +\sigma f.
\enn
Multiplying the above equation by $\pa_t u +\sigma u$ and integrating over $\Om_\rho^+$ yield
\be
\begin{aligned} \label{eq:st-1}
&\frac{1}{2}\frac{d}{dt}\| \pa_t u+\sigma u\|_{L^2(\Omega_\rho^+)}^2+ \frac{1}{2}\frac{d}{dt}\| \pa_t u^*+\sigma u^*\|_{L^2(\Omega_\rho^+)}^2+ \left( \nabla\cdot (\pa_t p+\sigma p),\,  \pa_t u+\sigma u\right)_{\Omega_\rho^+} \\
&+\left((\sigma+\hat \sigma) (\pa_t u+\sigma u),\,(\pa_t u+\sigma u) \right )_{\Omega_\rho^+}= \left(\pa_t f+\sigma f,\, \pa_t u+\sigma u\right)_{\Omega_\rho^+}.
\end{aligned}
\en
Since
\ben
\int_0^t \left((\sigma+\hat \sigma) (\pa_\tau u+\sigma u),\,(\pa_\tau u+\sigma u) \right )_{\Omega_\rho^+}\, d\tau\geq0,
\enn
integrating (\ref{eq:st-1}) from $0$ to $t$ and applying Green's first identity, we obtain
\be \label{eq:t-su}
\begin{aligned}
&\frac{1}{2}\| \pa_t u+\sigma u\|_{L^2(\Omega_\rho^+)}^2+ \frac{1}{2}\| \pa_t u^*+\sigma u^*\|_{L^2(\Omega_\rho^+)}^2- \int_0^t \left( (\pa_\tau p+\sigma p),\, \nabla ( \pa_\tau u+\sigma u)\right)_{\Omega_\rho^+}\,d\tau\\
\leq & \frac{1}{2}\| \pa_t u|_{t=0}\|_{L^2(\Omega_\rho^+)}^2+\frac{1}{2}\| \pa_t u^*|_{t=0}\|_{L^2(\Omega_\rho^+)}^2 +\int_0^t\left(\pa_\tau f+\sigma f,\, \pa_\tau u+\sigma u\right)_{\Omega_\rho^+}\,d\tau .
\end{aligned}
\en
Here we have used the fact that $u|_{t=0}=u^*|_{t=0}=0$.
We then apply $\pa_t $ to  the first equation of (\ref{eq:t-pml-b}) and (\ref{eq:t-pml-c}) and eliminate the term with $p^*$. This gives
\ben
\pa_{t}^2 p+\Lambda_1 \pa_t p +\Lambda_2 \nabla u+\nabla \pa_t u=0.
\enn
Multiplying the above equation by $\pa_t p +\sigma p$ and integrating over $\Om_\rho^+$ yield
\be \label{eq:t-p}
\begin{aligned}
&\frac{1}{2}\frac{d}{dt}\| \pa_t p+\sigma p\|_{L^2(\Omega_\rho^+)}^2+  \left(  \nabla ( \pa_t u+\sigma u),\, (\pa_t p+\sigma p)\right)_{\Omega_\rho^+} \\
&+\left((\Lambda_1-\sigma I) \pa_t p,\,(\pa_t p+\sigma p) \right )_{\Omega_\rho^+}=0.
\end{aligned}
\en
 Since $\sigma=\hat\sigma$, we have $\Lambda_1=\Lambda_2=\sigma I$ and $\Lambda_1-\sigma I=0$. Thus it follows from (\ref{eq:t-p}) that
\be \label{eq:t-sp}
    \frac{1}{2}\| \pa_t p+\sigma p\|_{L^2(\Omega_\rho^+)}^2+ \int_0^t \left(  \nabla ( \pa_\tau u+\sigma u),\, (\pa_\tau p+\sigma p)\right)_{\Omega_\rho^+} \,d\tau
    =\| \pa_t p|_{t=0}\|_{L^2(\Omega_\rho^+)}^2.
\en
Adding (\ref{eq:t-su}) and (\ref{eq:t-sp}) we get
\be
&&\frac{1}{2}\| \pa_t u+\sigma u\|_{L^2(\Omega_\rho^+)}^2+ \frac{1}{2}\| \pa_t u^*+\sigma u^*\|_{L^2(\Omega_\rho^+)}^2+ \frac{1}{2}\| \pa_t p+\sigma p\|_{L^2(\Omega_\rho^+)}^2 \nonumber\\
&\leq & \frac{1}{2}\| \pa_t u|_{t=0}\|_{L^2(\Omega_\rho^+)}^2
+\frac{1}{2}\| \pa_t u^*|_{t=0}\|_{L^2(\Omega_\rho^+)}^2 \nonumber\\
&&+\| \pa_t p|_{t=0}\|_{L^2(\Omega_\rho^+)}^2
+\int_0^t\left(\pa_\tau f+\sigma f,\, \pa_\tau u+\sigma u\right)_{\Omega_\rho^+}\,d\tau. \nonumber
\en
It follows from the compatibility conditions in (\ref{eq:t-pml-a})-(\ref{eq:t-pml-b}) and the initial conditions (\ref{eq:t-pml-f}) that
\begin{equation}
\pa_t u|_{t=0}=f|_{t=0}=0,\quad \pa_t u^*|_{t=0}=0, \quad \pa_t p|_{t=0}=0.
\end{equation}
Applying the Cauchy inequality with $\vep$, we have
\ben
&&\max_{0\leq t\leq  T} \left( \|\pa_t u +\sigma u \|_{L^2(\Omega_\rho^+)} +\|\pa_t p +\sigma p \|_{L^2(\Omega_\rho^+)} +\|\pa_t u^* +\sigma u^* \|_{L^2(\Omega_\rho^+)} \right)\\
\leq && C \int_0^T \| \pa_t f +\sigma f\|_{L^2(\Omega_\rho^+)}\, dt.
\enn
\end{proof}
The following lemma will be used to prove the stability of truncated PML problem (\ref{eq:t-pml}) which can be directly obtained from \cite[Lemma 3.2]{Chen12}.
\begin{lem} \label{lem:sg}
It holds that
\ben
\max_{0\leq t\leq T} \|  \sigma u \|_{L^2(\Omega_\rho^+)}\leq \max_{0\leq t\leq  T} \|  \pa_t u+\sigma u \|_{L^2(\Omega_\rho^+)}.
\enn
\end{lem}
The main result of this subsection is stated as follows.
\begin{thm} \label{thm:gg}
The solution $(u,p,u^*,p^*)$ to the truncated PML problem (\ref{eq:t-pml}) satisfies the stability estimate
\ben
&&\max_{0\leq t\leq  T} \left( \|\pa_t u  \|_{L^2(\Omega_\rho^+)} +\|\pa_t p  \|_{L^2(\Omega_\rho^+)} +\|\pa_t u^*  \|_{L^2(\Omega_\rho^+)}++\|\pa_t p^*  \|_{L^2(\Omega_\rho^+)} \right)\\
& \leq &   C \int_0^T \| \pa_t f +\sigma f\|_{L^2(\Omega_\rho^+)}\, dt.
\enn
\end{thm}
\begin{proof}
It follows from Lemma \ref{lem:sg} that
\ben
\max_{0\leq t\leq T}\| \pa_t u \|_{L^2(\Omega_\rho^+)} &\leq& \max_{0\leq t\leq T}\| \pa_t u +\sigma u \|_{L^2(\Omega_\rho^+)}+ \max_{0\leq t\leq T}\| \sigma u \|_{L^2(\Omega_\rho^+)} \\
&\leq&  2\max_{0\leq t\leq T}\| \pa_t u +\sigma u \|_{L^2(\Omega_\rho^+)}.
\enn
Similarly, we obtain
\ben
\max_{0\leq t\leq T}\| \pa_t p \|_{L^2(\Omega_\rho^+)} &\leq& 2\max_{0\leq t\leq T}\| \pa_t p +\sigma p \|_{L^2(\Omega_\rho^+)}.
\enn
Using (\ref{eq:t-pml-b}) and Lemma \ref{lem:sg},
\ben
\max_{0\leq t\leq T}\| \pa_t u^* \|_{L^2(\Omega_\rho^+)} =\max_{0\leq t\leq T}\| \sigma u \|_{L^2(\Omega_\rho^+)} \leq \max_{0\leq t\leq T}\| \pa_t u +\sigma u \|_{L^2(\Omega_\rho^+)}.
\enn
By (\ref{eq:t-pml-c}) and Lemma \ref{lem:sg}, one deduces
\ben
\max_{0\leq t\leq T}\| \pa_t p^* \|_{L^2(\Omega_\rho^+)} \;\; &\leq& 2\max_{0\leq t\leq T}\| \pa_t p^* +\sigma p^* \|_{L^2(\Omega_\rho^+)}\\
&=&2\max_{0\leq t\leq T}\| \pa_t p +\sigma p \|_{L^2(\Omega_\rho^+)}.\\
\enn
The desired estimate of Theorem \ref{thm:gg} follows from Theorem \ref{thm:sg} and the above estimates.
\end{proof}
\section {Convergence of PML method}
In this section, we shall prove convergence of the PML method.
First, we discuss the stability of an auxiliary problem for $(\tilde u, \tilde p, \tilde u^*, \tilde p^*) $ over  the PML layer $\Om_{PML}^+$. Consider
\be \label{eq:t-pml-l}
\left\{\begin{array}{lll}
\pa_t \tilde u+(\sigma +\hat \sigma)\tilde u +\sigma \tilde u^{*}+ \nabla \cdot \tilde p =0  \quad &&\mbox{in}\quad \Om_{PML}^+\times (0,T),\\
\pa_t \tilde{p}^{*}=-\nabla\tilde u \quad \pa_t \tilde{u}^{*}=-\sigma \tilde u \quad &&\mbox{in}\quad \Om_{PML}^+\times (0,T),\\
\pa_t \tilde{p}+\Lambda_1 \tilde p= \pa_t \tilde{p}^{*}+\Lambda_2 \tilde p^{*} \quad &&\mbox{in}\quad \Om_{PML}^+\times (0,T),\\
\tilde u =0 \quad &&\mbox{on}\quad (\Gamma_R^+ \cup \Gamma_0) \times (0,T),\\
\tilde u =\xi \quad &&\mbox{on}\quad  \Gamma_\rho^+ \times (0,T)\\
\tilde u|_{t=0} =\tilde p|_{t=0} =\tilde u^{*}|_{t=0}=\tilde p^{*}|_{t=0} \quad &&\mbox{in}\quad  \Om_{PML}^+.
\end{array}\right.
\en
Below we prove a trace lemma which will be used in proving the stability of the above auxiliary problem (\ref{eq:t-pml-l}).
\begin{lem} \label{lem:pml-stability}
Let $\xi \in  H^2(0,T;H^{1/2}_0(\Gamma_\rho^+))$. Then there exists a function $\zeta\in H^2(0,T;\\H^1(\Om_{PML}^+))$ such that $\zeta=0$ on $\Gamma_R^+\times(0,T)$, $\zeta=\xi$ on $\Gamma_\rho^+\times(0,T)$ and
\be
&&\|\pa_t^2 \zeta\|_{L^2(0,T;L^2(\Om_{PML}^+))}\leq C \rho^{1/2} \|\pa_t^2 \xi\|_{L^2(0,T;H^{-1/2}(\Gamma_\rho^+))},   \label{lem:tr-1}\\
&&\|\nabla \pa_t \zeta\|_{L^2(0,T;L^2(\Om_{PML}^+))}\leq C \rho^{-1/2} \|\pa_t \xi\|_{L^2(0,T;H^{1/2}_0(\Gamma_\rho^+))}. \label{lem:tr-2}
\en
\end{lem}

\begin{proof}
Expand $\xi(\theta,t)$ as follows
\ben
\xi(\theta,t)=\sum_{n=1}^{\infty} \xi_n(t)\sin n\theta,\quad \xi_n=\frac{2}{\pi}\int_{0}^{\pi} \xi(\theta, t)\sin n\theta \, d\theta.
\enn
Let $\chi_n\in C^{\infty}[R,\rho]$ such that $\chi_n(\rho)=1$, $0\leq\chi_n(r)\leq1$, $| \chi_n^{'}|\leq C \delta_n^{-1}$ for $r\in [R,\rho]$, and supp$(\chi_n)\subset(\rho-\delta_n,\rho)$, where $\delta_n=(\rho-R)/\sqrt{1+n^2}$, $n\in\Z$. Define the function
\ben
\zeta(t,r,\theta):=\sum_{n=1}^{\infty} \xi_n(t)\chi_n(r)\sin n\theta.
\enn
Then, it is clear that $\zeta=0$ on $\Gamma_R^+\times(0,T)$, $\zeta=\xi$ on $\Gamma_\rho^+\times(0,T)$.  It is obvious that
\ben
\int_0^T \|\pa_t^2\zeta\|_{L^2(\Om_{PML}^+)}^2 \,dt &=&\int_0^T \int_0^{\pi} \int_R^\rho \Big| \sum_{n=1}^{\infty} \xi_n^{''}(t)\chi_n(r)\sin n \theta \Big|^2 r \,dr d\theta dt\\
&=& \int_0^T \frac{\pi}{2} \sum_{n=1}^{\infty}  \int_R^\rho \big| \xi_n^{''}(t)\big|^2 \big| \chi_n(r)\big|^2 r \, dr dt \\
&\leq & \int_0^T \frac{\pi}{2} \sum_{n=1}^{\infty} \int_{\rho-\delta_n}^\rho \big| \xi_n^{''}(t)\big|^2  r \, dr dt \\
&\leq& \int_0^T \frac{\pi}{2}\rho \sum_{n=1}^{\infty} \delta_n \big| \xi_n^{''}(t)\big|^2   \, dt \\
&\leq& \int_0^T |\rho-R| \|\pa_t^2 \xi\|_{H^{-1/2}(\Gamma_\rho^+)}^2 \, dt \\
& \leq & C \rho \|\pa_t^2 \xi\|_{L^2(0,T;H^{-1/2}(\Gamma_\rho^+))}^2.
\enn
This proves (\ref{lem:tr-1}). Similarly, one can prove (\ref{lem:tr-2}).
\end{proof}
Theorem \ref{thm:au} below describes the stability of the solution to the problem (\ref{eq:t-pml-l}) in $\Om_{PML}^+$. It  can be easily proved by combining Lemmas \ref{lem:e-pml} and  \ref{lem:pml-stability} together with the Parseval identity. Since the proof is quite similar to  \cite[Theorem 4.3]{Chen09}, we omit the detailed proof.
\begin{lem} \label{thm:au}
Let $s_1=1/T$, $(\phi, \Phi, \phi^*, \Phi^*)$ be the solution of the PML problem (\ref{eq:t-pml-l}) in $\Om_{PML}^+$. Then
\ben
& &\|\pa_t \Phi\|_{L^2(0,T;L^2(\Om_{PML}^+))}+\|\nabla\cdot \Phi\|_{L^2(0,T;L^2(\Om_{PML}^+))}\\
&\leq & (1+\sigma T)^2 T \left( \rho  \|\pa_t^2 \xi\|_{L^2(0,T;H^{-1/2}(\Gamma_\rho^+))} + \rho^{-1} \|\pa_t \xi\|_{L^2(0,T;H^{1/2}_0(\Gamma_\rho^+))} \right ).
\enn
\end{lem}
We also need an estimate for the convolution proved in \cite[Lemma 5.2]{Chen09}.
\begin{lem} \label{lem:pml-2}
Let $g_1$, $g_2$ $\in $ $L^2(0,T)$. For any $\textnormal{Re}(s)=s_1>0$, it holds that
\be
\| g_1 \ast g_2\|_{L^2(0,T)}\leq e^{s_1 t} \left( \max_{-\infty<s_2<+\infty}|\mathscr{L}(g_1)(s_1+is_2)| \right ) \| g_2\|_{L^2(0,T)}.
\en
\end{lem}
The following result follows directly from the proof of  Lemma \ref{lem:t}.
\begin{lem} \label{lem:con}
Given $t\geq0$ and $\omega\in L^2(0,T;H^{1/2}_0(\Gamma_R^+))$ with the initial condition $\omega(\cdot,0)=0$, it holds that
\ben
-\textnormal{Re}\int_0^t  e^{-2s_1 \tau} \Big\langle \mathscr{T} \left(\int_0^\tau \omega(x,\eta) \,d \eta \right) ,\, \omega(x,\eta)   \Big \rangle\, d\tau\geq0.
\enn
\end{lem}

Now, we are ready to verify the exponential convergence of the time-domain PML method.
\begin{thm}  \label{convergence}
Let $(u,p)$ and $(\hat u, \hat p, \hat u^*, \hat p^*)$ be the solution of the problems (\ref{eqs:h-1}) and (\ref{t-pml}) with $s_1=T^{-1}$, respectively. Then
\ben
&& \max_{0\leq t \leq T}\left( \|u-\hat u \|_{L^2(\Om_R^+)}+\|p-\hat p \|_{L^2(\Om_R^+) }\right )\\
\leq &&
C\,(1+\sigma T)^2 \rho T^{3/2} e^{-\rho \hat \sigma(\rho)\left(1-\frac{R^2}{\rho^2}\right)}    \| \pa_t^2 \hat u\|_{L^2(0,T;H^{-1/2}(\Gamma_R^+))}   \\ &&+C\,(1+\sigma T)^2 \rho^{-1}T^{3/2} e^{-\rho \hat \sigma(\rho)\left(1-\frac{R^2}{\rho^2}\right)}    \| \pa_t \hat u\|_{L^2(0,T;H^{1/2}_0(\Gamma_R^+))},
\enn
where $C>0$ is a constant.
\end{thm}

\begin{proof}
By (\ref{eqs:h-1}) and (\ref{eq:t-pml-a})-(\ref{eq:t-pml-b}), it follows that
\be
\frac{\pa (u-\hat u)}{\pa t}+\nabla\cdot(p-\hat p)=0 \quad &&\mbox{in} \; \Om_R^{+}\times(0,T), \label{eq:c-1}\\
\frac{\pa (p-\hat p)}{\pa t}+\nabla (u-\hat u)=0  \quad &&\mbox{in} \; \Om_R^{+}\times(0,T). \label{eq:c-2}
\en
Multiplying both sides of (\ref{eq:c-1}) by a test function $v\in X_R$, using the DtN boundary condition (\ref{bc:dtn-1}) and Green's first formula, we obtain
\be \label{eq:c-3}
\begin{aligned}
&\Big (\frac{\pa (u-\hat u)}{\pa t},\,v  \Big)_{\Om_R^{+}}- (p-\hat p,\, \nabla v )_{\Om_R^{+}}-\Big\langle \mathscr{T}\Big(\int_0^t (u-\hat u)\, d\tau\Big),\, v \Big \rangle_{\Gamma_R^+} \\
=&\Big\langle \hat p \cdot \hat x+\mathscr{T}\Big(\int_0^t \hat u\, d\tau\Big),\, v \Big\rangle_{\Gamma_R^+}.
\end{aligned}
\en
Define
\ben
\om:=u-\hat u,\quad \om^*:=\int_0^t u-\hat u\,d\tau.
\enn
Taking $v=\om$ in (\ref{eq:c-3}) and applying (\ref{eq:c-3}) with $p-\hat p|_{t=0}=0$, we have
\be \label{eq:c-4}
\frac{1}{2}\frac{d}{dt}\Big( \| \om \|_{L^2(\Om_R^+)}^2 +  \| \nabla \om^* \|_{L^2(\Om_R^+)}^2  \Big)-\langle \mathscr{T}(\om^*),\, \om \rangle_{\Gamma_R^+}= \Big\langle \hat p \cdot \hat x+\mathscr{T}\Big(\int_0^t \hat u\, d\tau\Big),\, \om \Big\rangle_{\Gamma_R^+}.
\en
Denote the spaces
\ben
&&X(0,T;\Om_R^+):=\Big\{v\in L^{\infty}(0,T;L^2(\Om_R^+)),\, v^*=\int_0^t v\, dt \in L^{\infty}(0,T;H^1(\Om_R^+))\Big\},\\&&Y(0,T;\Gamma_R^+):=\Big\{\phi:\int_0^T \langle \phi,\,v \rangle_{\Gamma_R^+}\,dt<\infty, \forall\; v\in X(0,T;\Om_R^+) \Big\}.
\enn
It is clear that $X(0,T;\Om_R^+)$ and $Y(0,T;\Gamma_R^+)$  are Banach spaces with the norms, respectively
\be
&&\|v\|_{X(0,T;\Om_R^+)}=\sup_{0\leq t\leq T}\Big( \| v\|_{L^2(\Om_R^+)}^2  + \| \nabla v^*\|_{L^2(\Om_R^+)}^2 \Big)^{1/2},\\
&&\|\phi \|_{Y(0,T;\Gamma_R^+)}=\sup_{v\in X(0,T;\Om_R^+)}\frac{\Big| \int_0^T \langle \phi,\, v \rangle_{\Gamma_R^+} \,dt \Big| }{\|v\|_{X(0,T;\Om_R^+)}}. \label{eq:con-y}
\en
Multiplying both sides of (\ref{eq:c-4}) by $e^{-2s_1 t}$ and then integrating from $0$ to $t$.  Since $\om|_{t=0}=\om^*|_{t=0}=0$, taking the real part of the resulting identity and using Lemma \ref{lem:con} and trace theorem, we obtain
\ben
\|e^{-s_1  t} \om\|_{{X(0,T;\Om_R^+)}}^2\leq& C \Big\| e^{-s_1  t} \Big (\hat p \cdot \hat x+\mathscr{T}\Big(\int_0^t \hat u\, d\tau \Big) \Big )\Big \|_{Y(0,T;\Gamma_R^+)}  \, \|e^{-s_1  t} \om\|_{Y(0,T;\Gamma_R^+)}\\
\leq & C \| e^{-s_1  t}  \Big (\hat p \cdot \hat x+\mathscr{T}\Big(\int_0^t \hat u\, d\tau\Big)\Big ) \Big\|_{Y(0,T;\Gamma_R^+)}  \, \|e^{-s_1  t} \om\|_{X(0,T;\Om_R^+)}.
\enn
Hence, by taking $s_1\rightarrow 0$
\be \label{eq:con-4}
\sup_{0\leq t\leq T}\Big( \| \om\|_{L^2(\Om_R^+)}^2  + \| \nabla \om^*\|_{L^2(\Om_R^+)}^2 \Big)  \leq C\Big \| \hat p \cdot \hat x+\mathscr{T}\Big(\int_0^t \hat u\, d\tau \Big)   \Big\|_{Y(0,T;\Gamma_R^+)}.
\en
It is clear that $\mathscr{T}\Big(\int_0^t \hat u\, d\tau \Big)  =-\tilde {\hat p}\cdot \hat x$ on $\Gamma_R^+$, where $\tilde {\hat p}$ defines the PML extension of $\hat p$. Hence, in order to estimate $\Big\| \hat p \cdot \hat x+\mathscr{T}\Big(\int_0^t \hat u\, d\tau \Big)   \Big\|_{Y(0,T;\Gamma_R^+)}$, it suffices to estimate $\|(\hat p-\tilde{\hat p})\cdot \hat x\|_{Y(0,T;\Gamma_R^+)}$.
Since any function $v\in X(0,T;\Om_R^+)$ can be extended into $\Om_{PML}^+\times(0,T)$ such that $v=0$ on $\Gamma_\rho^+\times(0,T)$ and $\|v\|_{X(0,T;\Om_R^+)}\leq C \|v\|_{X(0,T;\Om_{PML}^+)}$, it follows by (\ref{eq:con-y}) that
\be
\begin{aligned}  \label{eq:con-5}
\|(\hat p-\tilde{\hat p})\cdot \hat x\|_{Y(0,T;\Gamma_R^+)}  =&& \sup_{v\in X(0,T;\Om_R^+)}\frac{\Big| \int_0^T \langle \phi,\, v \rangle_{\Gamma_R^+} \,dt \Big| }{\|v\|_{X(0,T;\Om_R^+)}}\\
\leq && \sup_{v\in X(0,T;\Om_R^+)}\frac{\Big| \int_0^T \langle \phi,\, v \rangle_{\Gamma_R^+} \,dt \Big| }{\|v\|_{X(0,T;\Om_{PML}^+)}}.
\end{aligned}
\en
For any $v\in X(0,T;\Om_{PML}^+)$ it has that $v=0$ on $\Gamma_\rho^+$, and then, by divergence theorem,
\be \label{eq:con-6}
\int_0^T \langle (\hat p- \tilde{\hat p})\cdot \hat x,\, v \rangle_{\Gamma_R^+}\, dt= \int_0^T \Big[(\nabla\cdot(\hat p-\tilde{\hat{p}}),\,v)_{\Om_{PML}^+}+(\hat p-\tilde{\hat{p}},\,\nabla v)_{\Om_{PML}^+} \Big]\,dt.
\en
Now,  it follows that,  for  any $v\in X(0,T;\Om_{PML}^+)$, by the definition of $v^*$ and the initial condition $\hat p-\tilde{\hat{p}}|_{t=0} =0$,
\be \label{eq:con-7}
\begin{aligned}
\Big | \int_0^T (\hat p- \tilde{\hat p},\, \nabla v)_{\Om_{PML}^+}\, dt \Big | = &&(\hat p- \tilde{\hat p},\,\nabla v^*)_{\Om_{PML}^+}\Big|_{0}^T-\int_0^T (\pa_t(\hat p- \tilde{\hat p}),\, \nabla v^*)_{\Om_{PML}^+}\, dt \\
\leq && 2 \max_{0\leq t\leq T}\| \nabla v^*\|_{L^2(\Om_{PML}^+)} \int_0^T \| \pa_t(\hat p- \tilde{\hat p})\|_{L^2(\Om_{PML}^+)}\, dt.
\end{aligned}
\en
Combining (\ref{eq:con-5}), (\ref{eq:con-6}), (\ref{eq:con-7}) and using the Cauchy-Schwartz inequality, we have
\ben
\|(\hat p-\tilde{\hat p})\cdot \hat x\|_{Y(0,T;\Gamma_R^+)} \leq C \int_0^T  \| \nabla \cdot (\hat p- \tilde{\hat p})\|_{L^2(\Om_{PML}^+)}+ \| \pa_t(\hat p- \tilde{\hat p})\|_{L^2(\Om_{PML}^+)} \,dt.
\enn
This together with (\ref{eq:con-4}) leads to
\ben
&&\sup_{0\leq t\leq T}\Big( \| \om\|_{L^2(\Om_R^+)}^2  + \| \nabla \om^*\|_{L^2(\Om_R^+)}^2 \Big) \\
\leq && C \int_0^T  \| \nabla \cdot (\hat p- \tilde{\hat p})\|_{L^2(\Om_{PML}^+)}+ \| \pa_t(\hat p- \tilde{\hat p})\|_{L^2(\Om_{PML}^+)} \,dt.
\enn
Let $(\tilde{\hat u}, \tilde{\hat p}, \tilde{\hat u}^*, \tilde{\hat p}^*)$ be the PML extension of $(\hat u,  \hat p, \hat u^*, \hat p^*)$.  Then, $(\hat u-\tilde{\hat{u}}, \hat p- \tilde{\hat p}, \hat u^*-\tilde{\hat{u}}^*, \hat p^*- \tilde{\hat p}^*)$ satisfies the problem (\ref{eq:t-pml-l})  with $\xi=-\tilde{\hat u}|_{\Gamma_\rho^+}$. It follows by Theorem \ref{thm:au}  and  Cauchy-Schwartz inequality that
\be
\begin{aligned} \label{eq:con-8}
&\sup_{0\leq t\leq T}\Big( \| \om\|_{L^2(\Om_R^+)}^2  + \| \nabla \om^*\|_{L^2(\Om_R^+)}^2 \Big)\\
\leq& \quad  C (1+\sigma T)^2 T^{3/2} \left( \rho  \|\pa_t^2 \tilde{\hat{u}}\|_{L^2(0,T;H^{-1/2}(\Gamma_\rho^+))} + \rho^{-1} \|\pa_t \tilde{\hat{u}}\|_{L^2(0,T;H^{1/2}_0(\Gamma_\rho^+))} \right ).
\end{aligned}
\en
Now we estimate each term on the right hand side of the above inequality. Since $\tilde{\hat u}$ is the PML extension of $\hat u$ in the time-domain for $r>R$,  it can be written as
\ben
\tilde{\hat u}(r,\theta, t)
=\sum_{n=1}^{\infty} \left [ \mathscr{L}^{-1}\Big(\frac{K_n'(s\tilde{r})}{K_n(sR)}\Big)*\hat u_n(R,t)\right]\,\sin{n\theta},
\enn
where $\hat u_n(R,t)=\frac{2}{\pi}\int_0^{\pi} \hat  u(R,\theta,t)\sin n\theta \,d\theta$. Since $\hat u_n(R,0)=0$, we have
\ben
\pa_t \tilde{\hat u}=\sum_{n=1}^{\infty} \left [ \mathscr{L}^{-1}\Big(\frac{K_n'(s\tilde{r})}{K_n(sR)}\Big)*\pa_t \hat  u_n(R,t)\right]\,\sin{n\theta}.
\enn
Then, since $s\tilde{\rho}=s\rho+\hat\sigma\rho$, by Lemmas \ref{lem:pml-1} and   \ref{lem:pml-2}, we know that for any $s_1>0$
\ben
&&\|\pa_t \tilde {\hat u}\|_{L^2(0,T;H^{1/2}_0(\Gamma_\rho^+))}^2=\int_0^T \|\pa_t \tilde u  \|_{H^{1/2}_0(\Gamma_\rho^+))}^2\,dt \\
=&& \int_0^T\frac{\pi}{2} \rho \sum_{n=1}^{\infty} (1+n^2)^{\frac{1}{2}} \left [ \mathscr{L}^{-1}\left(\frac{K_n'(s\tilde{\rho})}{K_n(sR)}\right)*\pa_t \hat u_n(R,t)\right]^2  \,dt \\
=&& \frac{\pi}{2} \rho \sum_{n=1}^{\infty} (1+n^2)^{\frac{1}{2}} \Big \| \mathscr{L}^{-1}\left(\frac{K_n'(s\tilde{\rho})}{K_n(sR)}\right)*\pa_t \hat u_n(R,t)\Big \|^2_{L^2(0,T)}\\
\leq && \frac{\pi}{2} \rho\, e^{2s_1 T} \sum_{n=1}^{\infty} (1+n^2)^{\frac{1}{2}} \max_{-\infty<s_2<+\infty}\left|\frac{K_n'(s\tilde{\rho})}{K_n(sR)}\right|^2\| \pa_t \hat  u_n(R,t)\|^2_{L^2(0,T)}\\
\leq && \frac{\rho}{R}\, e^{2s_1 T}  \max_{-\infty<n<+\infty} \max_{-\infty<s_2<+\infty}\left|\frac{K_n'(s\tilde{\rho})}{K_n(sR)}\right|^2\| \pa_t \hat u\|^2_{L^2(0,T;H^{1/2}_0(\Gamma_R^+))}\\
\leq && \frac{\rho}{R}\, e^{2s_1 T} e^{-2\rho \hat \sigma(\rho)\left(1-\frac{R^2}{\rho^2}\right)}  \| \pa_t \hat u\|^2_{L^2(0,T;H^{1/2}_0(\Gamma_R^+))}.\\
\enn
This implies that
\be
\|\pa_t \tilde{ \hat u}\|_{L^2(0,T;H^{1/2}_0(\Gamma_\rho^+))} \leq Ce^{-\rho \hat \sigma(\rho)\left(1-\frac{R^2}{\rho^2}\right)}  \| \pa_t \hat u\|_{L^2(0,T;H^{1/2}_0(\Gamma_R^+))}. \label{eq:con-9}
\en
Analogously, we obtain
\be
\|\pa_t^2 \tilde {\hat u}\|_{L^2(0,T;H^{-1/2}_0(\Gamma_\rho^+))} \leq C e^{-\rho \hat \sigma(\rho)\left(1-\frac{R^2}{\rho^2}\right)}  \| \pa_t^2 \hat u\|_{L^2(0,T;H^{-1/2}(\Gamma_R^+))}. \label{eq:con-10}
\en
Combining (\ref{eq:con-8}) with (\ref{eq:con-9}) and (\ref{eq:con-9}), we complete the proof.

\end{proof}

\begin{rem}
Theorem \ref{convergence} illustrates that the exponential convergence of error between the PML solution and the original solution can be achieved by enlarging the absorbing parameter $\sigma$  or the thickness $\rho-R$ of the PML layer .
\end{rem}

\begin{rem}
We remark that the results in this paper can be easily extended to the Neumann boundary condition imposed on $\Gamma_0$. In the Neumann case, one should expand the solutions in terms of cosine functions in the Laplace domain and change correspondingly the solution spaces.  However, we don't know how to extend the approach to  the case of the impedance boundary condition.
\end{rem}

\section{Numerical implementation}

In this section, we will present two numerical examples to demonstrate the convergence of the PML method. The PML equations are discretized by the finite element method in space and finite difference in time. The computations are carried out by the software FreeFEM. 

 Multiply (\ref{eq:t-pml-a})-(\ref{eq:t-pml-c}) with test functions $v\in H^1_{0}(\Om_\rho^+)$, $q\in L^2(\Om_\rho^+)$, $v^*\in H^1_{0}(\Om_\rho^+)$, $q^*\in L^2(\Om_\rho^+)$, respectively. The weak formulation of system (\ref{eq:t-pml}) reads as follow: find  $u\in L^2(0,T;  H^1_{0}(\Om_\rho^+))$, $p\in L^2(0,T; L^2(\Om_\rho^+))$, $u^*\in L^2(0,T; L^2(\Om_\rho^+))$, $p^*\in L^2(0,T; L^2(\Om_\rho^+))$ such that
\begin{subequations}  \label{eq:t-pml-weak}
\begin{align}
&\left(\pa_t \tilde u,v\right)_{\Om_\rho^+}+((\sigma +\hat \sigma)\tilde u, v)_{\Om_\rho^+} +(\sigma \tilde u^{*}, v)_{\Om_\rho^+}- ( \tilde p, \nabla  v)_{\Om_\rho^+} =(f,v)_{\Om_\rho^+},  \\
&(\pa_t \tilde{p},q)_{\Om_\rho^+}+(\Lambda_1 \tilde p,q)_{\Om_\rho^+}-( \pa_t \tilde{p}^{*},q)_{\Om_\rho^+}-(\Lambda_2 \tilde p^{*},q)_{\Om_\rho^+}=0, \\
&(\pa_t \tilde{u}^{*}, v^{*})_{\Om_\rho^+}-(\sigma \tilde u, q^{*})_{\Om_\rho^+}=0,  \\
&(\pa_t \tilde{p}^{*},q^{*})_{\Om_\rho^+} + (\nabla \tilde u,  q^{*})_{\Om_\rho^+}=0,\\
&\tilde u =0 \quad \mbox{on}\; \pa D \cup \Gamma_0 \times (0,T),\\
&\tilde u =0 \quad \mbox{on}\;  \Gamma_\rho^+ \times (0,T),\\
&\tilde u|_{t=0} =\tilde p|_{t=0} =\tilde u^{*}|_{t=0}=\tilde p^{*}|_{t=0} =0\quad \mbox{in}\;  \Om_\rho^+.
\end{align}
\end{subequations}
Solutions of the weak form (\ref{eq:t-pml-weak}) will be numerically solved by an implicit finite difference in time and a finite element method in space.  Let $\{t_0, t_1,...,t_N\}$ be a partition of the time interval $[0, T]$ and $\delta_ t^n = t_{n+1}-t_n$ be the $n$-th time-step size. Let $\mathcal M_h$ be a regular triangulation of $\Om_\rho^+$. We assume the elements $K \in \mathcal M_h$ may have one curved edge align with $\Gamma_0\cup\Gamma_\rho^+$ so that $\Om_\rho^+ = \cup_{ K\in \mathcal M_h }K$. As usual, we shall use the most simple continuous finite elements in the computation. The solutions $u$, $p$, $u^{*}$ and $p^{*}$  will be approximated in  the  finite element space $P_1$ for piecewise linear functions, namely, the standard Taylor-Hood finite element for the velocity-pressure variables, satisfying the inf-sup condition. Denote by $P_0$ the finite element space for piecewise constant functions. 
Define the spaces $L_h$, $\tilde V_h$, $V_h$ and $W_h$ as
\ben
L_h=\{\sigma \in L^2 (\Om_{\rho}^+)\,|\,\forall K\in \mathcal T_h,\, \sigma_{|K}\in P_0\},\\
 \tilde V_h=\{u\in H^1_{\star}(\Om_{\rho}^+)\,|\,\forall K\in \mathcal T_h,\, u_{|K}\in P_1\},\\
 V_h=\{u\in H^1(\Om_{\rho}^+)\,|\,\forall K\in \mathcal T_h,\, u_{|K}\in P_1\},\\
 W_h=\{p \in L^2 (\Om_{\rho}^+)\,|\,\forall K\in \mathcal T_h,\, \sigma_{|K}\in P_1\}.
 \enn
Let $(u_h^n, p_h^n, u_h^{*n}, p_h^{*n}) \in \tilde V_h\times (W_h)^2\times V_h\times (W_h)^2$
 be an approximation of $u(t_n), p(t_n), u^{*}(t_n)$ and $p^{*}(t_n)$ at the time point $t_n$. The approximated solution at $t_{n+1}$, which we denote as  $(u_h^{n+1}, p_h^{n+1}, u_h^{*n+1}, p_h^{*n+1}) \in \tilde V_h\times (W_h)^2\times V_h\times (W_h)^2$, will be obtained by the following typical temporal scheme:
\begin{figure}[htb]
\centering
\includegraphics[scale=0.6]{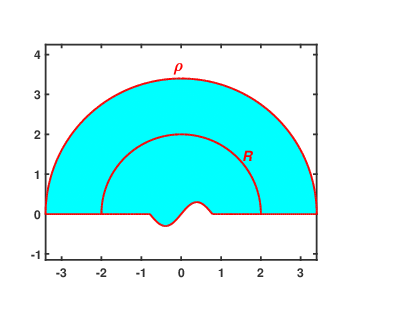}
\caption{Geometry of the computational domain where $\{x\in \Omega: R<|x|<\rho\}$ is the PML layer.
} \label{fig:mesh}
\end{figure}
\begin{subequations}
\begin{align*}
&\left( \frac{\tilde u_h^{n+1}-\tilde u_h^{n}}{\delta_t^n}
,v\right)_{\Om_\rho^+}+\left((\sigma +\hat \sigma)\tilde u_h^{n+1}, v\right)_{\Om_\rho^+} +\left(\sigma \tilde u_h^{*n+1}, v\right)_{\Om_\rho^+} - \left( \tilde p_h^{n+1}, \nabla  v\right)_{\Om_\rho^+} =\left(f^{n+1},v\right)_{\Om_\rho^+},  \\
&\left(\frac{\tilde{p}_h^{n+1}-\tilde{p}_h^{n}}{\delta_t^n},q\right)_{\Om_\rho^+}+\left(\Lambda_1 \tilde p_h^{n+1},q\right)_{\Om_\rho^+}-\left( \pa_t \tilde{p}^{*n+1}_h,q\right)_{\Om_\rho^+}-\left(\Lambda_2 \tilde p^{*n+1}_h,q\right)_{\Om_\rho^+}=0, \\
&\left( \frac{\tilde u_h^{*n+1}-\tilde u_h^{*n}}{\delta_t^n}, v^{*}\right)_{\Om_\rho^+}-\left(\sigma \tilde u_h^{n+1}, q^{*}\right)_{\Om_\rho^+}=0,  \\
&\left(\frac{\tilde{p}_h^{*n+1}-\tilde{p}_h^{*n}}{\delta_t^n},q^{*}\right)_{\Om_\rho^+} + \left( \nabla \tilde u_h^{n+1},  q^{*}\right)_{\Om_\rho^+}=0,
\end{align*}
\end{subequations}
where $f^{n+1}=f(t_{n+1})$ and the Dirichlet boundary condition is imposed on $\partial \Omega_\rho^+$.

In the following numerical examples, we suppose $D= \emptyset $. The local rough surface is given by 
\ben
h(x)= \left\{\begin{array}{lll}
0,  && x\in(-\infty, -\frac{\pi}{4}), \\
0.3\sin(4x), && x\in[ -\frac{\pi}{4}, \frac{\pi}{4}],\\
0,  && x\in(\frac{\pi}{4}, \infty).
\end{array}\right.
\enn

$\textbf{Example 1}$ We consider  a time harmonic source term over the local rough surface. In the computation, we take $R=2$, $\rho=3$ and set $\sigma=\hat\sigma=10$. A mesh of $9245$ vertices, $510$ edges and $18128$ triangles is adopted and the terminal time is set at $t=8$. The time harmonic source is supposed to be given by
\ben
f(x,t):=\frac{e^{\frac{-|x-x_0|^2}{2\eta}}}{\sqrt{2\pi}\;\eta}\, \sin(2t),\; \eta=0.1,
\enn

\begin{figure}[htb]
\centering
\subfigure[$t=2$]{
\includegraphics[scale=0.2]{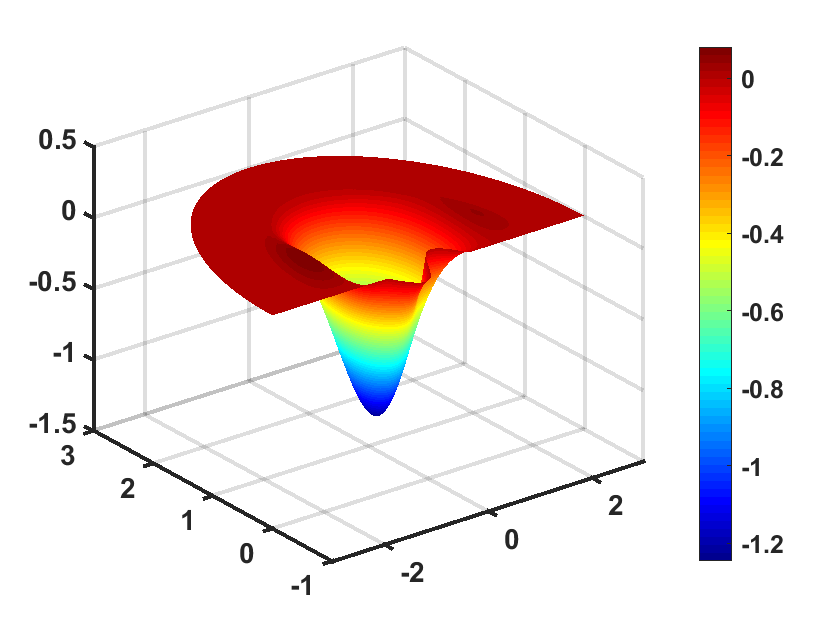}
}
\subfigure[$t=4$ ]{
\includegraphics[scale=0.2]{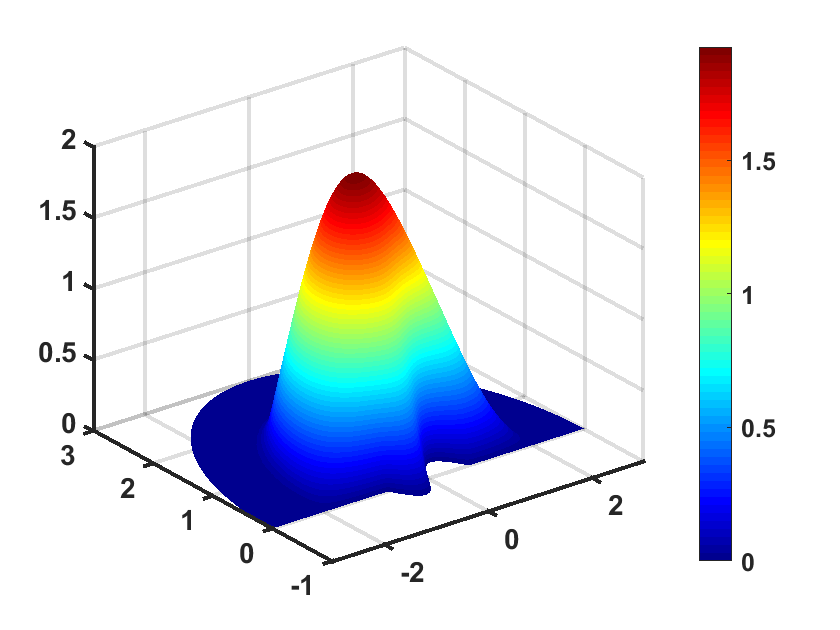}
}
\subfigure[$t=8$]{
\includegraphics[scale=0.2]{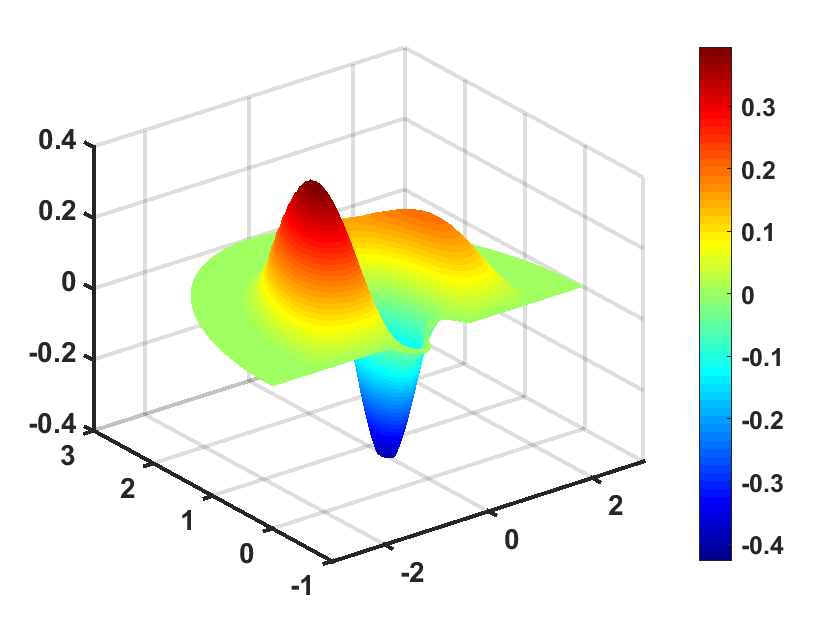}

}
\caption{Numerical solutions excited by  a point source over a  local rough surface at time  $t=2,4,8$, respectively.} \label{fig:sint}
\end{figure}

\noindent where the excitation frequency is $\omega=2$ and the location  of source is at $x_0=(0,0.5)$.  In  Figure $\ref{fig:sint}$, we show the numerical solution at $t=2,4,8$, respectively. It is observed that the waves are almost attenuated in the PML layer without reflections from the interface between the physical domain and PML layer.

To validate convergence of the PML method, we compute the relative error
\ben
E_{rel}:=\frac{\| u^{num}-u^{exa}  \|_{L^{\infty}(0,T;L^\infty(\Om_\rho^+))}}{\|u^{exa}\|_{L^{\infty}(0,T;L^\infty(\Om_\rho^+))}},
\enn
where $u^{exa}$ represents the numerical solution with relatively larger absorbing parameters $\sigma$, $\hat \sigma$ and with a larger thickness $\rho-R$ of the PML layer. Note that analytical solutions are not available in general.

\begin{figure}[htb]
\centering
\includegraphics[scale=0.5]{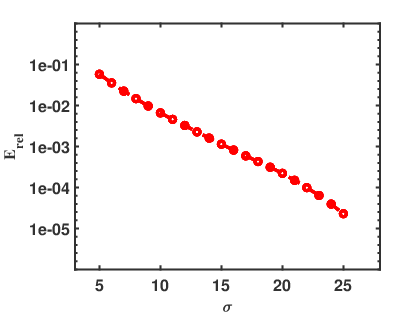}
\caption{ Relative error $E_{rel}$  versus  PML absorbing parameter $\sigma$   for fixed PML thickness $\rho-R=1$. The vertical axis is logarithmically scaled.
} \label{fig:sigma}
\end{figure}

\begin{figure}[htb]
\centering
\includegraphics[scale=0.5]{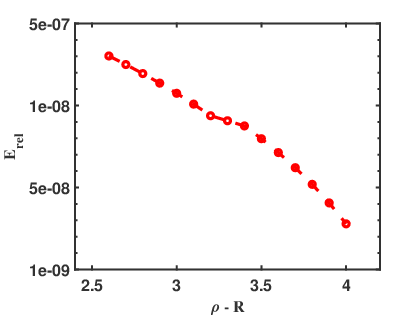}
\caption{ Relative error $E_{rel}$  versus  PML thickness  $\rho-R$ for fixed PML absorbing parameter $\sigma=25$.  The vertical axis is logarithmically scaled.
} \label{fig:thick-sin2t}
\end{figure}


Figures \ref{fig:sigma} and \ref{fig:thick-sin2t} show the decaying behavior of  the relative error $E_{rel}$  as the  PML absorbing parameter  $\sigma$ or the PML thickness $\rho-R$ increases. In Figure \ref{fig:sigma} we take the PML absorbing parameter $\sigma$ varying between $5$  and $25$ and fix  the PML layer thickness at $\rho -R=1$.  Since the vertical axis is logarithmically scaled,   the  dashed lines indicate that the relative error $E_{rel}$ decays exponentially as $\sigma$ increases for a fixed layer thickness. In Figure $\ref{fig:thick-sin2t}$ we display the relative error versus PML thickness $\rho-R$ for fixed absorbing $\sigma=25$. We take the PML thickness $\rho-R$ ranging from $2.6$ to $4$. It is obvious that relative error $E_{rel}$ decays exponentially as $\rho-R$ increases for a fixed absorbing parameter.

$\textbf{Example 2}$ In this example, we consider a non-harmonic source term, that is not compactly supported with respect to the time-variable. We take $R=2$, $\rho=3.4$ and set $\sigma=\hat \sigma =25$.  A mesh of $11502$ vertices,  $553$ edges and $22599$ triangles is applied and the final time is set at $t=10$.  The source is of the form
\ben
f(x,t):=\frac{e^{\frac{-|x-x_0|^2}{2\eta}}}{\sqrt{2\pi}\;\eta}\, t,\; \eta=0.1,
\enn
which is located at $(0,0.5)$.  We present in Figure \ref{fig:t} the numerical solutions at time $t=3, 7, 10$, respectively. It shows that the source is excited all the time and rare reflection occurs at the interface $\Gamma_R^+$.

\begin{figure}[htb]
\centering
\subfigure[$t=3$]{
\includegraphics[scale=0.2]{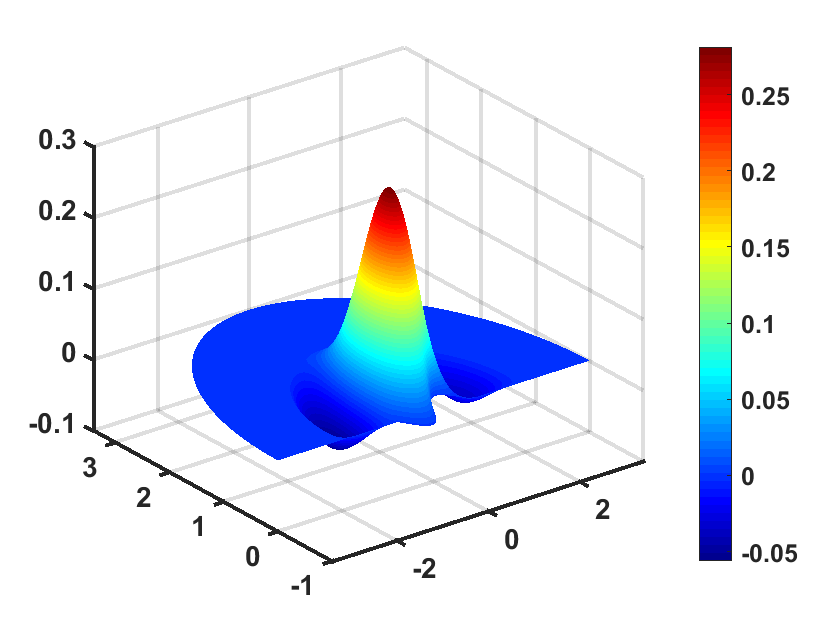}
}
\subfigure[$t=7$ ]{
\includegraphics[scale=0.2]{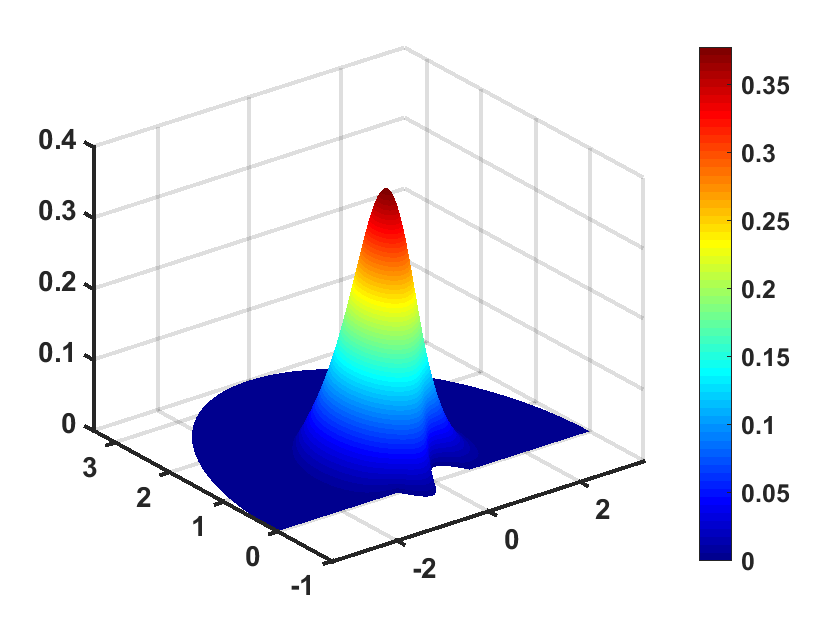}
}
\subfigure[$t=10$]{
\includegraphics[scale=0.2]{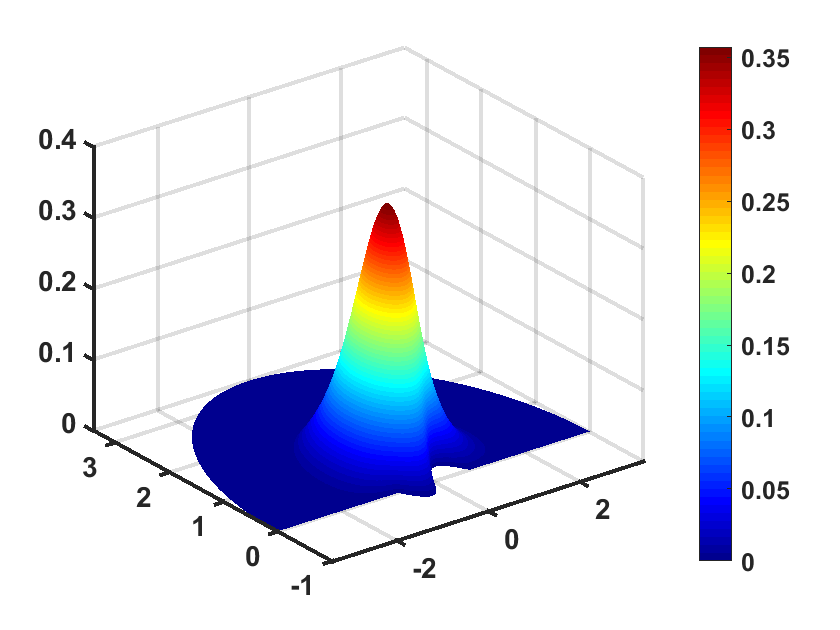}

}
\caption{Numerical solutions for a point source over a  local rough surface at time  $t=3,7,10$, respectively.} \label{fig:t}
\end{figure}

\begin{figure}[htb]
\centering
\includegraphics[scale=0.5]{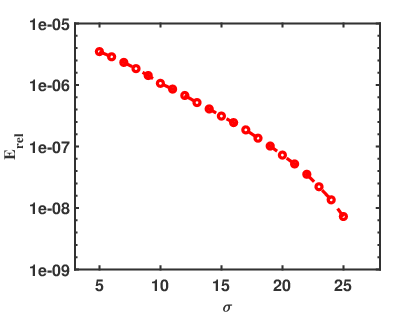}
\caption{ Relative error $E_{rel}$  versus  PML absorbing parameter $\sigma$   for fixed PML thickness $\rho-R=1.4$. The vertical axis is logarithmically scaled.
} \label{fig:sigma-t}
\end{figure}
\begin{figure}[htb]
\centering
\includegraphics[scale=0.5]{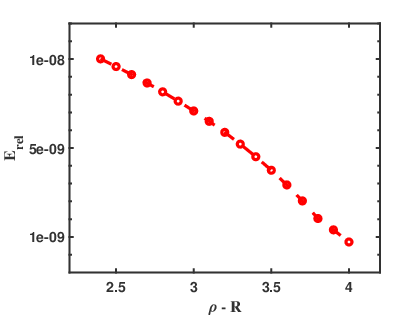}
\caption{ Relative error $E_{rel}$  versus  PML thickness  $\rho-R$ for fixed PML absorbing parameter $\sigma=25$. The vertical axis is logarithmically scaled.
} \label{fig:thick-t}
\end{figure}

Figures \ref{fig:sigma-t} and \ref{fig:thick-t} show the convergence of the relative error $E_{rel}$ versus one of the two PML parameters $\sigma$ and $\rho-R$. In Figure \ref{fig:sigma-t}, we present the relative error $E_{rel}$ versus the PML absorbing parameter $\sigma$ changing from $5$ to $25$ for the fixed PML thickness $\rho-R=1.4$. As in \textbf{Example 1}, we observe that $E_{rel}$ decays exponentially as $\sigma$ increases. In Figure \ref{fig:thick-t} we display the relative error versus the PML thickness $\rho-R$ for a fixed absorbing $\sigma=25$. We take the PML thickness $\rho-R$ changing from $2.4$ to $4$. It is obvious that the relative error $E_{rel}$ decays exponentially as $\rho-R$ increases.

\section{Conclusion}

In this paper we study the PML-method to  the time domain acoustic  scattering problem  over a locally rough surface. We proved well-posedness of the scattering problem  and the PML formulation. The long time stability of the PML formulations is obtained by the energy method. The exponential convergence of  the PML solution is proved and it can be realized by either enlarging the PML absorbing parameter or the thickness of the layer. The convergence results are verified by two numerical examples.

\newpage

\begin{appendices}
\section{Laplace transform}
For any $s=s_1+is_2$ with $s_1>0$, $s_2\in \R$, we define the Laplace transform of $u$ as
\ben
u_L(s)=\mathscr{L}(u)(s):=\int_0^\infty e^{-st}u(t)\, dt.
\enn
Some properties of the Laplace transform and its inversion are listed  as follows:
\be
\mathscr{L}(\frac{du}{dt})&=&s u_L-u(0),\\
\mathscr{L}(\frac{d^2u}{dt^2})&=&s^2 u_L- su(0) -\frac{du}{dt}(0),\\
\int_0^t u(\tau)\, d\tau&=&\mathscr{L}^{-1}(s^{-1}u_L(s)). \label{eq:l-3}
\en
It can be verified from the inverse Laplace transform that
\ben
u(t)=\mathscr{F}^{-1}(e^{s_1t} \mathscr{L}(u)(s_1+is_2)).
\enn
where $\mathscr{F}^{-1}$ denotes the inverse Fourier transform with respect to $s_2$.

Below is the Parseval or Plancherel identity for the Laplace transform  (see \cite[equation (2.46)]{Cohen})
\begin{lem}
\textnormal{(Parseval or Plancherel Identity)} If $u_L=\mathscr{L}(u)$ and $v_L=\mathscr{L}(v)$,then
\be \label{PI}
\frac{1}{2\pi}\int_{-\infty}^{+\infty} u_L(s)\,v_L(s) \,ds_2=\int_0^{\infty} e^{-2s_1 t} u(t)\,v(t)\,dt,
\en
for all $s_1>s_0$, where $s_0$ is abscissa of convergence for Laplace transform of $u$ and $v$.
\end{lem}

The following lemma refers to \cite[Theorem 43.1]{Treves}.
\begin{lem} \label{lem:a}
let $\breve{\omega}(s)$ denote a holomorphic function in the half complex plane $\textnormal{Re} (s)>\sig_0$ for some $\sigma_0\in\R$, valued in the Banach space $\E$. The following conditions are equivelent:
\item[(1)] there is a distribution $\om \in \mathcal{D}_+^{'}$ whose Laplace transform is equal to  $\breve\omega(s)$, where $\mathcal{D}^{'}_+(\E)$ is the space of distributions on the real line which vanish identically in the open negative half-line;
\item[(2)] there is a $\sig_1$ with $\sig_0\leq\sig_1<\infty$ and an integer $m\geq0$ such that for all complex numbers $s$ with $\textnormal{Re} (s)>\sig_1$  it holds that $\|\breve{\omega}(s)\|_{\E}\leq C(1+|s|)^m$.
\end{lem}

\section{Sobolev spaces}

For a bounded domain $D$ with Lipschitz continuous boundary $\partial D$, define the Sobolev space
\ben
H(\mbox{div},\,D):=\{u\in L^2(D): \nabla \cdot u\in L^2(D)\}.
\enn
which is a Hilbert space with the norm
\ben
\|u\|_{H(\mbox{div},\,D)}=\left (\|u\|^2_{L^2 (D)}+\|\nabla\cdot u\|^2 _{L^2(D)}\right)^{1/2}.
\enn

\n The first Green formula takes the form
\be
(\nabla\cdot u,\, v)_{D}+(u,\, \nabla v)_{D}=\langle u\cdot n,\, v \rangle_{\partial D} \quad \mbox{for all}\quad u, v \in H(\mbox{div},\,D)
\en
where $(\cdot, \cdot)_{D}$ and $\langle \cdot ,\cdot \rangle_{\partial D}$ denote the $L^2$-inner product on $D$ and the dual pairing product between $H^{-1/2}(\pa D)$ and $H^{1/2}(\pa D)$, respectively.

Let $\Gamma_R^+$  be defined as in Section 2.  For any $u\in C^{\infty}_0(\Gamma_R^+)$, we have the Fourier series expansion
\ben
u(R,\theta)=\sum_{n=1}^{\infty} a_n \sin n \theta, \quad a_n=\frac{2}{\pi}\int_0^\pi u(R,\theta) \sin n \theta\, d\theta.
\enn
Define the trace function space $H^p_0(\Gamma_R^+)$ as
\ben
H^p_0(\Gamma_R^+):=\left \{u\in L^2(\Gamma_R^+): \|u\|_{H^p(\Gamma_R^+)}:= \left(\sum_{n=1}^{\infty}(1+n^2)^p|a_n|^2\right)^{1/2}<+\infty. \right\}.
\enn

\section{Modified Bessel functions}

We look for solutions to the  Helmholz equation
\ben
\Delta u(x)-s^2u(x)=0,
\enn
in the form of
\ben
u(x)=y(sr)e^{in\theta}, \quad n=0, \pm1, \pm2, \cdots,
\enn
where $(r,\theta)$ are cylindrical coordinates. It is obvious that $y(r)$ is a solution of the ordinary equation
\ben
\frac{d^2y}{dr^2}+\frac{1}{r}\frac{dy}{dr}-\left(1+\frac{n^2}{r^2}\right)y=0.
\enn
The modified Bessel functions of order $\nu$, which we denote by $K_{\nu}(z)$, are solutions to the differential equation
\ben
z^2 \frac{d^2y}{dz^2}+z\frac{dy}{dz}-(z^2+\nu^2)y=0.
\enn
$K_\nu(z)$ satisfies the following asymptotic behavior as $|z|\rightarrow\infty$
\ben
K_\nu(z)\sim\left( \frac{\pi}{2z} \right)^{1/2}e^{-z}.
\enn

The following estimates for the modified Bessel function $K_\nu(z)$  have been proved in \cite[Lemma 2.10 and 5.1]{Chen09} .
\begin{lem} \label{lem:mbf-1}
Let $s=s_1+i s_2$ with $s_1>0$, $s_2\in\R$. It holds that
\ben
-\textnormal{Re}\Big(\frac{K_n^{'}(sR)} {K_n(sR)}\Big)\geq0.
\enn
\end{lem}

\begin{lem} \label{lem:pml-1}
 Suppose $\nu \in \R$, $s=s_1+i s_2$ with $s_1>0$, $s_2\in\R$,  $\rho_1>\rho_2>0$ and $\tau>0$. It holds that
\ben
\frac{|K_{\nu}(s\rho_1+\tau)|}{|K_{\nu}(s\rho_2)|}\leq e^{-\tau (1-\rho_2^2/\rho_1^2)}.
\enn
\end{lem}

\end{appendices}

\section*{Acknowledgements}
G. Hu is partially supported by the National Natural Science Foundation of China (No. 12071236) and the Fundamental Research Funds for Central Universities in China (No. 63213025).

\end{document}